\documentclass[preprint,12pt]{elsarticle}

\usepackage{amsmath,amssymb,amsthm,a4wide}
\usepackage{url}
\usepackage{latexsym}
\usepackage{indentfirst}
\usepackage{graphicx}
\usepackage{placeins}
\usepackage{booktabs}
\usepackage{algorithm}
\usepackage{algorithm}
\usepackage{algpseudocode}
\algrenewcommand\algorithmicprocedure{\textbf{function}}
\usepackage{multirow}
\usepackage{bbm}
\usepackage{enumerate}
\usepackage{bm}
\usepackage{subfig}
\usepackage{graphicx,color}
\usepackage{diagbox}
\usepackage{epstopdf}

\theoremstyle{plain}
\newtheorem{theorem}{Theorem}
\newtheorem{lemma}[theorem]{Lemma}
\newtheorem{assumption}[theorem]{Assumption}

\newtheorem{prop}[theorem]{Proposition}
\newtheorem{cor}[theorem]{Corollary}

\theoremstyle{definition}

\theoremstyle{remark}
\newtheorem{remark}{Remark}

\newtheoremstyle{cited}%
  {3pt}% (space above)
  {3pt}% (space below)
  {\itshape}% (body font)
  {}% (indent amount)
  {\bfseries}% {theorem head font}
  {.}% {punctuation after theorem head}
  {.5em}% {space after theorem head}
  {\thmname{#1} \thmnumber{#2} \thmnote{\normalfont#3}}% {theorem head spec}

\theoremstyle{cited}

\newcommand{\bl}{{(l)}}
\newcommand{\blm}{{(l-1)}}

\definecolor{eggplant}{RGB}{180,33,147}

\newcommand{\boxl}[1]{[#1]}

\newcommand\bbRm[2]{\mathbb{R}^{#1 \times #2}}

\newcommand{\bvec}[1]{\bm{#1}}

\newcommand{\vx}{\bvec{x}}
\newcommand{\vy}{\bvec{y}}

\newcommand{\eg}{\textit{e.g.}~}

\newcommand{\bbN}{\mathbb{N}}

\newcommand{\bbR}{\mathbb{R}}

\newcommand{\ccA}{\mathcal{A}}

\newcommand{\ccC}{\mathcal{C}}

\newcommand{\ccI}{\mathcal{I}}
\newcommand{\ccJ}{\mathcal{J}}

\newcommand{\ccP}{\mathcal{P}}

\journal{Journal of Computational Physics}
\begin{document}
% \maketitle

% \begin{frontmatter}
\begin{frontmatter}

\title{Generative Modeling via Hierarchical Tensor Sketching}

\author[label1]{Yifan Peng}
\author[label1]{Yian Chen}
\address[label1]{Committee on Computational and Applied Mathematics, University of Chicago}
% \address[Yifan Peng, Yian Chen, Yuehaw Khoo]{Department of Statistics, University of Chicago}

\author[label2]{E. Miles Stoudenmire}
\address[label2]{Center for Computational Quantum Physics, Flatiron Institute}
% \address[E. Miles Stoudenmire]{Center for Computational Quantum Physics, Flatiron Institute}
\author[label1,label3]{Yuehaw Khoo}
\address[label3]{Department of Statistics,
University of Chicago}
\begin{abstract}
We propose a hierarchical tensor-network approach for approximating high-dimensional probability density via empirical distribution. This leverages randomized singular value decomposition (SVD) techniques and involves solving linear equations for tensor cores in this tensor network. The complexity of the resulting algorithm scales linearly in the dimension of the high-dimensional density. An analysis of estimation error demonstrates the effectiveness of this method through several numerical experiments. 
\end{abstract}

\begin{keyword}Tensor Train, Generative Modeling, Randomized Algorithm, Hierarchical Tensor Decomposition
\end{keyword}
\end{frontmatter}

\section{Introduction}
Density estimation is one of the most fundamental problems in statistics and machine learning. Let $p^*$ be a probability density function on a $d$-dimensional space $\bbR^d$. The task of density estimation is to estimate $p^*$ based on a set of independently and identically distributed (i.i.d) samples $\{\vy^i\}_{i=1}^N$ drawn from the density. More precisely, our goal is to estimate $p^*$ from the sample empirical distribution,

\begin{align}
\label{eq:empirical_density}
    \hat{p}(\vx) = \frac{1}{N}\sum_{i=1}^N \delta_{\vy^i}(\vx),\ \text{where } \vx\in\bbR^d,
\end{align}
where $\delta_{\vy^i}$ is the $\delta$-measure supported on $\vy^i\in\bbR^d$. 

Traditional density estimators including the histograms \cite{hist1,hist2} and kernel density estimators \cite{kde1,kde2} (KDEs) typically perform well in low-dimensional settings. While it is possible to engineer kernels for the above methods that work in high-dimensional cases, 
recently it has been common to use neural network-based approaches that can learn features for high-dimensional density estimations. This includes autoregressive models \cite{autoreg1,autoreg2,autoreg3}, generative adversarial networks (GANs) \cite{gan}, variational autoencoders (VAEs) \cite{vae}, and normalizing flows \cite{nf1,nf2,nf3,nf4}. 

Alternatively, several works have proposed to approximate the density function in low-rank tensor networks, in particular the tensor train (TT) format \cite{tt} (known as matrix product state (MPS) in the physics literatures \cite{mps1,mps2}). Tensor-network represented distribution can efficiently generate independent and identically distributed (i.i.d.) samples by applying conditional distribution sampling \cite{dolgov2020approximation} to the obtained TT format. Furthermore, one can take advantage of the efficient TT format for fast downstream task applications, such as sampling, conditional sampling, computing the partition function, solving committor functions for large scale particle systems \cite{chen2023committor}, etc. 

In this work, we propose a randomized linear algebra based algorithm to estimate the probability density in the form of a hierarchical tensor-network given an empirical distribution, which further extends the MPS/TT \cite{tt-sketch} and tree tensor-network structure \cite{tang2022generative}. Hierarchical format applications play a critical role in statistical modeling by allowing for the representation of complex systems and the interactions between different local components. Various hierarchical structures have been proposed in the past with active applications in different fields, such as Bayesian modeling \cite{lawson2012bayesian,lee2011cognitive}, conditional generations \cite{saibaba2012efficient}, Gaussian process computations \cite{chen2021scalable,ambikasaran2015fast,geoga2020scalable,chen2023scalable}, optimizations \cite{shin2019hierarchical,chen2020multiscale,karsanina2018hierarchical} and high-dimensional density modeling \cite{genton2018hierarchical,cao2019hierarchical}. The hierarchical structure can effectively handle spatial random field such as Ising models on lattice where the sites can be clustered hierarchically. Similar to \cite{tt-sketch}, we use sketching technique to form a parallel system of linear equations with constant size with respect to the dimension $d$ where the coefficients of the linear system are moments of the empirical distribution. By solving $d\log_2{d}$ number of linear equations, one can obtain a hierarchical tensor approximation to the distribution.

\subsection{Prior work}
In the context of recovering low-rank tensor trains (TTs), there are two general types of input data. The first type involves the assumption that a $d$-dimensional function $p$ can be evaluated at any point and it aims to recover $p$ with a limited number of evaluations (typically polynomial in $d$). Techniques such as TT-cross \cite{oseledets2010tt}, DMRG-cross \cite{savostyanov2011fast}, TT completion \cite{steinlechner2016riemannian}, and generalizations \cite{wang2017efficient,khoo2017efficient} for tensor rings are under this category.

In the second scenario, where only samples from distribution are given, \cite{ttde1,ttde2,ttde3} propose a method to recover the density function by minimizing some measure of error, such as the Kullback-Leibler (KL) divergence, between the empirical distribution and the MPS/TT ansatz. However, due to the nonlinear parameterization of a TT in terms of the tensor components, optimization approaches can be prone to local minima issues. A recent work \cite{tt-sketch} proposes and analyzes an algorithm that outputs an MPS/TT representation of $\hat{p}$ to estimate $p^*$, by determining each MPS/TT core in parallel via a perspective that views randomized linear algebra routine as a method of moments.

Although our method shares some similarities with tensor compression techniques in the current literature and can be utilized for compressing a large exponentially sized tensor into a low-complexity hierarchical tensor, our focus is primarily on an estimation problem instead of an approximation problem. Specifically, we aim to estimate the ground truth distribution $p^*$ that generates the empirical distribution $\hat{p}$ in terms of a TT, within a generative modeling context. Assuming the existence of an algorithm $\ccA$ that can take any $d$-dimensional function $p$ and output its corresponding TT as $\ccA(p)$, then one would expect that such $\ccA$ could minimize the following differences:
$$  p^* - \ccA(\hat{p}) = \underbrace{p^* - \ccA(p^*)}_{\text{approximation error}} + \quad \underbrace{ \ccA(p^*) -\ccA(\hat{p })}_{\text{estimation error}} $$
In generative modeling context, the empirical distribution $\hat{p}$ is affected by sample variance, leading to variance in the TT representation $\ccA(\hat{p})$ and, consequently, the estimation error. Our method is designed to minimize this estimation error. 

In contrast, methods based on singular value decomposition (SVD) \cite{tt} and randomized linear algebra \cite{oseledets2010tt,savostyanov2011fast,shi2021parallel,generalized_nystorm} aim to compress the input function $p$ as a TT such that $\ccA(p)\approx p$. Using such methods in the statistical learning context can result in an estimation error of $\ccA(p^*)-\ccA(\hat{p})\approx p^*-\hat{p}$ which grows exponentially with the number of variables $d$ for a fixed number of samples.

\subsection{Contributions}
Here we summarize our contributions.
\begin{enumerate}
    \item We propose an optimization free method to construct a hierarchical tensor-network for approximating a high-dimensional probability density via empirical distribution in $O(N d\log d)$ complexity. The proposed method is suitable for representing a distribution coming from a spatial random field. 

    \item We analyze the estimation error in terms of Frobenious norm for high-dimensional tensor. We show that the error decays like $O(\frac{c^{\log d}}{\sqrt{N}})$ for some constant $c$ that depends on some easily computable norm of the reconstructed hierarchical tensor network.
\end{enumerate}

\subsection{Organization}
The paper is organized as follows. In Section~\ref{sec:notation} and Section~\ref{sec:tensor diagrams}, we introduce notations and tensor diagrams in this paper, respectively. In Section~\ref{sec:hierarchical_section3}, we present our algorithm for density estimation in terms of hierarchical tensor-network. In Section~\ref{sec:error analysis}, we analyze the estimation error of the algorithm. In Section~\ref{sec:numerical} we perform numerical experiments in several 1D and 2D Ising-like models.  Finally, we conclude in Section~\ref{sec:conclusion}.

\subsection{Notations}\label{sec:notation}

For an integer $n\in\bbN$, we define $\boxl{n}=[1,2,\cdots,n]$. For two integers $m,n\in\bbN$ where $n> m$, we use the \texttt{MATLAB} notation $m:n$ to denote the set $[m,m+1,\cdots,n]$.

In this paper, we will extensively work with vectors, matrices and tensors. We use boldface lower-case letters to denote vectors (\eg $\bm{x}$). We denote matrices and tensors by capital letters (\eg $A$). Let $\mathcal{I}$ and $\mathcal{J}$ be two sets of indices, we use $A(\mathcal{I},\mathcal{J})$ to denote a submatrix of $A$ with rows indexed by $\mathcal{I}$ and columns indexed by $\mathcal{J}$. We also use $A(\mathcal{I},:)$ or $A(:,\mathcal{J})$ to denote a set of rows or columns, respectively. Similarly, we use $A(\mathcal{I},:,:)$ to denote a set of first dimensional array indexed by $\ccI$ for a three-dimensional tensor $A$. To simplify the notation for matrix multiplication, we use $:$ to represent the multiplied dimension. For instance, $A(\ccI,:)B(:,\mathcal{K})=\sum_{j\in\ccJ}A(\ccI,j)B(j,\mathcal{K})$.

Besides, we generalize the notion of Frobenius norm $\|\cdot \|_F$ such that for a $d$-dimensional tensor $p$, it is defined as
\begin{equation}
    \| p \|_F = \sqrt{\sum_{x_1\in[n_1],\ldots,x_d\in[n_d]} p(x_1,\ldots,x_d)^2}.
\end{equation}

When working with discrete high-dimensional functions, it is convenient to reshape the function into a matrix. We call these matrices as \textit{unfolding matrices}. For a $d$-dimensional density $p:\ \boxl{n_1}\times\boxl{n_2}\times\cdots\times\boxl{n_d}\rightarrow\bbR$ for some $n_1,\cdots,n_d\in\bbN$, we define the $k$-th unfolding matrix of a $d$-dimensional function by $p(x_1,\cdots,x_k;x_{k+1},\cdots,x_d)$ or $p(\vx_{1:k},\vx_{k+1:d})$, which can be viewed as a matrix of shape $\prod^k_{i=1} n_i\times \prod^d_{i=k+1} n_{i}$. \par 

\subsection{Tensor diagrams}\label{sec:tensor diagrams}
To illustrate the method, we use tensor diagram frequently in the paper. In tensor diagrams, a tensor is represented by a node, where the dimensionality of the tensor is indicated by the number of its incoming legs. We use circle to represent a node in the following tensor diagram. In Figure~\ref{fig:dg_demo}(A), a three-dimensional tensor $A$ and a two-dimensional matrix $B$ are depicted in the tensor diagram, which can be treated as two functions $A(i_1,i_2,i_3)$ and $B(j_1,j_2)$, respectively. Besides, we use bold line to illustrate 
that the specific dimension has a larger size compared to other dimensions. Thus the number of $i_3$ is essentially larger than the number of $i_1$ or $i_2$ and the same property holds for $j_1$ compared to $j_2$ in Figure~\ref{fig:dg_demo}(A).\par 

In addition, we could use tensor diagram to represent tensor multiplication or tensor contraction, which is able to provide a convenient way to understand this operation. In Figure~\ref{fig:dg_demo}(B), leg $i_3$ of $A$ is joined with leg $j_1$ of $B$ and this supposes implicitly that the two legs have the same size, represented by a new notation $k$. This corresponds to the computation: $\sum_{k}A(i_1,i_2,k)B(k,j_2)=C(i_1,i_2,j_2)$. Besides, reshaping is another significant operation of tensor. In Figure~\ref{fig:dg_demo}(C), we demonstrate reshaping by combining the first two legs of tensor $A$ ($i_1$ and $i_2$) into a new leg with a size equivalent to the product of the sizes of $i_1$ and $i_2$. This procedure converts a three-dimensional tensor into a two-dimensional matrix. This is illustrated in Figure~\ref{fig:dg_demo}(C). More specifically, 
to emphasize the large size of the combined dimensions, we use a bold line in the right diagram of Figure~\ref{fig:dg_demo}(C), with index $i_{1:2}$, which follows the notation introduced in Section~\ref{sec:notation}.

Further background about tensor diagram and tensor notations could be found in \cite{orus2014advances}. An example is depicted in Figure~\ref{fig:h_tt} and this would be explained in detail in the next section.

\begin{figure}[!htb]
    \centering
    \subfloat[Tensor diagrams of 3-dimensional $A$ and 2-dimensional $B$.]{{\includegraphics[width=0.48\textwidth]{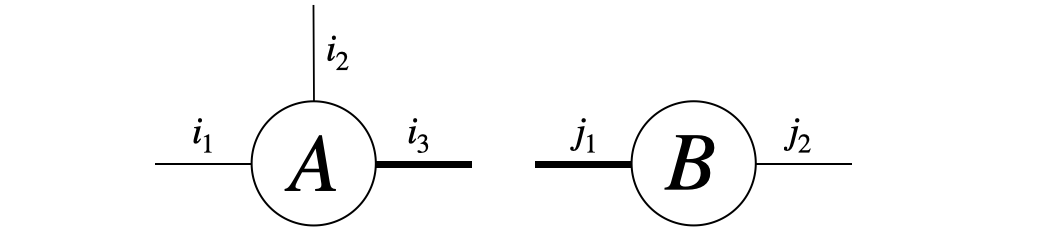}}}
    \quad 
    \subfloat[Tensor diagram of tensor multiplication. ]{{\includegraphics[width=0.48\textwidth]{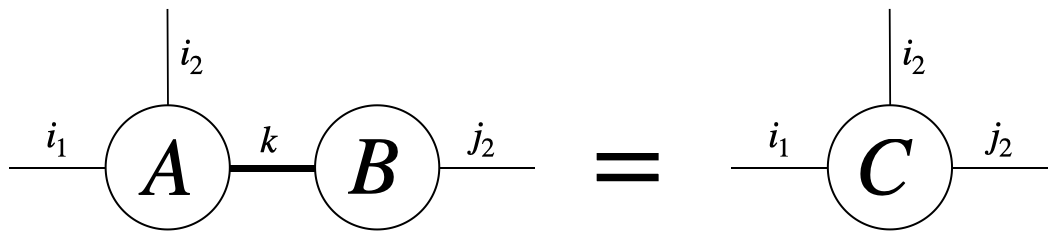}}}\\
    \subfloat[Tensor diagrams of reshaping.]{{\includegraphics[width=0.55\textwidth]{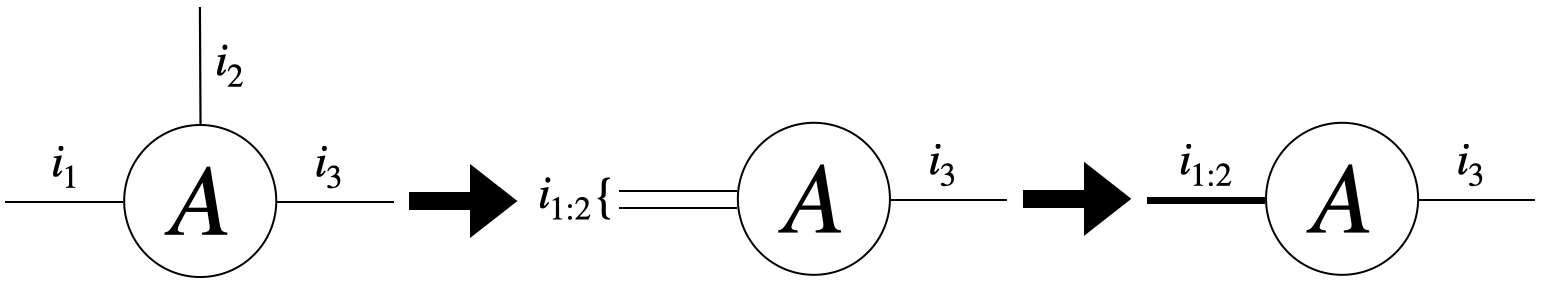}}}
    \caption{Tensor diagrams examples. }\label{fig:dg_demo}    
\end{figure}

\section{Density estimation via hierarchical tensor-network}\label{sec:hierarchical_section3}
In this section, we estimate a discrete $d$-dimensional density $p$ from the empirical distribution in terms of a hierarchical tensor-network. For convenience, we assume $d=2^L$, and we partition the dimension in a binary fashion to construct our hierarchical tensor-network, although the proposed method can be easily generalized to arbitrary choice of $d$. We assume $p:\ \boxl{n_1}\times\boxl{n_2}\times\cdots\times\boxl{n_d}\rightarrow\bbR$ for some $n_1,\cdots,n_d\in\bbN$. Given $N$ i.i.d samples from the density $p$, our objective is to approximate the density using a hierarchical tensor representation. 
\subsection{Main idea}\label{sec:main idea}
We start with a simple example to illustrate the main idea. Suppose a density only has two variables $p(x_1,x_2)$ and $\text{rank}(p) = r$. We can obtain a decomposition of $p$ via the following procedure: We first determine the column and row spaces of $p$ as $p_1(x_1,\gamma_1)$ and $p_2(x_2, \gamma_2)$, $\gamma_1,\gamma_2\in [r]$, then we determine a matrix $G \in \bbRm{r}{r}$ via solving the following equation for all $x_1,x_2$:
\begin{equation}
p_1(x_1,:)G p_2(x_2,:)^T = p(x_1,x_2),
\end{equation}
which in turn provides a low-rank approximation to $p$. For high-dimensional cases, we can recursively apply this binary partition to determine a hierarchically low-rank tensor-network. 

More specifically, let $l=0,\ldots,L$ be the levels of the hierarchical tensor-network. At each level $l$, we cluster the dimensions $[d]$ into $2^l$ parts and $k\in [2^l]$ is the index of nodes per level. Clusters of each level are defined as $\mathcal{C}^{(l)}_k:=(k-1)(\frac{2^L}{2^l})+1:k(\frac{2^L}{2^l})$ for $ k\in[2^l]$, and thus they partition $[d]$. The main idea is to solve for the nodes $G_k^{(l)}$ at each level via a set of core defining equations: 
\begin{small}
\begin{align}\label{eq:high d cdse}
p_1^{(1)}(\ccI_1^{(1)},:) G_1^{(0)}(:,:,\tilde \gamma_{1}^{(0)}) p_2^{(1)}(\mathcal{I}^{(1)}_2,:)^T&=  p_1^{(0)}(\ccI_1^{(1)},\ccI_2^{(1)},\tilde{\gamma}_1^{(0)}),\ \tilde{\gamma}_1^{(0)}\in [1], \nonumber \\  
p_{2k-1}^{(2)}(\mathcal{I}^{(2)}_{2k-1},:) G_{k}^{(1)}(:,:,\tilde \gamma_{k}^{(1)})p_{2k}^{(2)}(\mathcal{I}^{(2)}_{2k},:)^T &=  p_{k}^{(1)}(\mathcal{I}^{(2)}_{2k-1},\mathcal{I}^{(2)}_{2k}, \tilde \gamma_{k}^{(1)}), \tilde \gamma_k^{(1)} \in [\tilde{r}^{(1)}], k \in  [2], \nonumber \\ 
& \vdots \nonumber \\  
p_{2k-1}^{(l)}(\mathcal{I}^{(l)}_{2k-1},:) G_{k}^{(l-1)}(:,:,\tilde \gamma_{k}^{(l-1)})p_{2k}^{(l)}(\mathcal{I}^{(l)}_{2k},:)^T &=  p_{k}^{(l-1)}(\mathcal{I}^{(l)}_{2k-1},\mathcal{I}^{(l)}_{2k},\tilde \gamma_{k}^{(l-1)}), \tilde \gamma_k^{(l-1)} \in [\tilde{r}^{(l-1)}], k \in  [2^{l-1}], \nonumber \\
& \vdots  \nonumber\\ 
p_{2k-1}^{(L)}(\mathcal{I}^{(L)}_{2k-1},:)  G_{k}^{(L-1)}(:,:,\tilde \gamma_{k}^{(L-1)}) p_{2k}^{(L)}(\mathcal{I}^{(L)}_{2k},:)^T &=  p_{k}^{(L-1)}(\mathcal{I}^{(L)}_{2k-1},\mathcal{I}^{(L)}_{2k},\tilde \gamma_{k}^{(L-1)}), \tilde \gamma_{k}^{(L-1)} \in [\tilde{r}^{(L-1)}],k \in  [2^{L-1}], \nonumber \\ 
G_k^{(L)}&=p_k^{(L)},k \in  [2^{L}]. \nonumber \\ 
\end{align}
\end{small}The above equations \eqref{eq:high d cdse} are the key equations in this paper. Here $p_1^{(0)} := p$ and $\mathcal{I}_k^{(l)} = \prod_{i\in \mathcal{C}_k^{(l)}} [n_i]$ is the set of all possible values for $\vx_{\mathcal{C}_k^{(l)}}$ for $k\in[2^l],l\in[L]$. Similarly, we denote $\mathcal{J}_{k}^{(l)} = \prod_{i\in [d]\backslash\mathcal{C}_{k}^{(l)}} [n_i]$ as the set of all possible values for $\vx_{[d]\backslash\mathcal{C}_{k}^{(l)}}$ for every $k,l$. For convenience, we denote $G_k^{(L)}:=p_k^{(L)}$ for $k\in [2^L]$ in the last level. The core equations could be depicted directly via the tensor diagram in Figure~\ref{fig:ceq}, following the pattern in Section~\ref{sec:tensor diagrams}. Here clusters $\mathcal{C}_k^\bl$ satisfy the following relationship between two levels: $\mathcal{C}_k^\bl=\mathcal{C}_{2k-1}^{(l+1)}\cup \mathcal{C}_{2k}^{(l+1)}$, and furthermore, $\ccI_k^\bl=\ccI_{2k-1}^{(l+1)}\times \ccI_{2k}^{(l+1)}$ holds for $k\in[2^l],l\in[L-1]$. Then $p_k^{(l)}(\ccI_k^{(l)},\tilde{\gamma}_k^{(l)}),\tilde{\gamma}_k^{(l)}\in[\tilde{r}^{(l)}]$ is obtained by reshaping the first two indices of $p_{k}^{(l)}(\mathcal{I}^{(l+1)}_{2k-1},\mathcal{I}^{(l+1)}_{2k},\tilde \gamma_{k}^{(l)})$. We use $p_k^\bl$ to represent the two-dimensional form unless we emphasize the third index of it. 

\begin{figure}[!htb]
    \centering   \subfloat{{\includegraphics[width=0.7\textwidth]{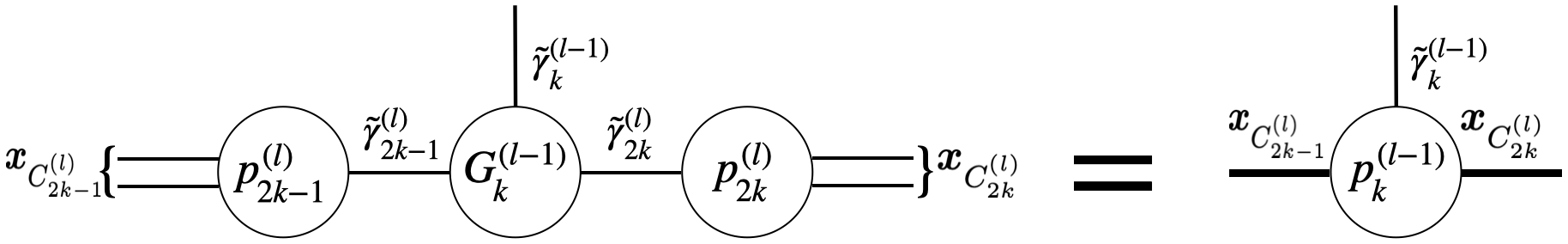}}}
    \caption{Tensor diagram of core defining equations \eqref{eq:high d cdse}. In order to go from left hand side to right hand side, we use the reshaping operation in Figure~\ref{fig:dg_demo}(C) to reshape the dimension ($\vx_{\ccC_{4k-3}^{(l+1)}},\vx_{\ccC_{4k-2}^{(l+1)}}$) of 
    $p_{2k-1}^{(l)}$ into $\vx_{\ccC_{2k-1}^{(l)}}$ and reshape the  dimension ($\vx_{\ccC_{4k-1}^{(l+1)}},\vx_{\ccC_{4k}^{(l+1)}}$) of
    $p_{2k}^{(l)}$ into $\vx_{\ccC_{2k}^{(l)}}$.}\label{fig:ceq}
\end{figure}

We now state an assumption for the unfolding density $p_k^\bl$  that guarantees that each equation in \eqref{eq:high d cdse} would have solution $G_k^\bl$. 
\begin{assumption}\label{assumption:range}
We assume $\text{Range}\left(p^{(l)}_{2k-1}\right)\supset \text{Range}\left(p_k^{(l-1)}(\mathcal{I}_{2k-1}^{(l)};\mathcal{I}_{2k}^{(l)},\tilde \gamma_k^{(l-1)})\right)$ and \\ $\text{Range}\left(p_{2k}^\bl\right) 
\supset \text{Range}\left(p_k^{(l-1)}(\mathcal{I}_{2k}^{(l)};\mathcal{I}_{2k-1}^{(l)},\tilde \gamma_k^{(l-1)})\right)$, respectively, for $k\in [2^{l-1}]$ and $l\in [L]$. 
\end{assumption}
Assumption~\ref{assumption:range} in turn gives us the following theorem, which guarantees a hierarchical tensor-network representation of $p$ in terms of $\{G^{(l)}_k\}_{k,l}$, as depicted by a tensor diagram in Figure~\ref{fig:h_tt}.
\begin{theorem}\label{thm:Grepre}
Suppose Assumption~\ref{assumption:range} holds, then $\{\{G_k^\bl\}_{k=1}^{2^l}\}_{l=0}^{L}$ in \eqref{eq:high d cdse} gives a hierarchical tensor-network representation of $p$.  
\end{theorem}
\begin{proof}
Based on Assumption~\ref{assumption:range}, we have \text{Range}($p_1^{(0)}(\ccI_1^{(1)};\ccI_2^{(1)},\tilde{\gamma}_1^{(0)}))\subset \text{Range}(p_1^{(1)})$ and \\  \text{Range}($p_1^{(0)}(\ccI_2^{(1)};\ccI_1^{(1)},\tilde{\gamma}_1^{(0)}))\subset \text{Range}(p_2^{(1)})$, therefore the equation defining $G^{(0)}_1$ in \eqref{eq:high d cdse} admits a solution, which comes from taking the pseudo-inverse of $p_1^{(1)},p_2^{(1)}$, i.e. \begin{equation*}
G^{(0)}_1(:,:,\tilde \gamma_1^{(0)})  = \left(p_1^{(1)}\right)^\dagger p_1^{(0)}(:,:,\tilde \gamma_1^{(0)}) \left(p_2^{(1)T}\right)^\dagger, \tilde{\gamma}_1^{(0)} \in [1],
\end{equation*}
Now we already construct a representation of $p_1^{(0)} = p$ based on $p^{(1)}_1, p^{(1)}_2$ and $G^{(0)}_1$, after solving the equation involving $G^{(0)}_1$ in \eqref{eq:high d cdse}. Furthermore, we could express $p^{(1)}_1$ and $p^{(1)}_2$ in terms of $G^{(1)}_1, p^{(2)}_1, p^{(2)}_2$ and  $G^{(1)}_2, p^{(2)}_3, p^{(2)}_4$ by solving the equations for $G^{(1)}_1$ and $G^{(1)}_2$ in  \eqref{eq:high d cdse} (which again admits solutions due to Assumption~\ref{assumption:range}). By recursing this procedure on $ p^{(2)}_1, p^{(2)}_2,  p^{(2)}_3, p^{(2)}_4$, and the lower level $\{\{p^{(l)}_k\}_{k=1}^{2^l}\}_{l=3}^L$, we obtain a tensor-network representation of $p$ in terms of $\{\{G^{(l)}_k\}_{k=1}^{2^l}\}_{l=0}^L$.

\end{proof}

\begin{figure}[!htb]
    \centering   \subfloat{{\includegraphics[width=0.4\textwidth]{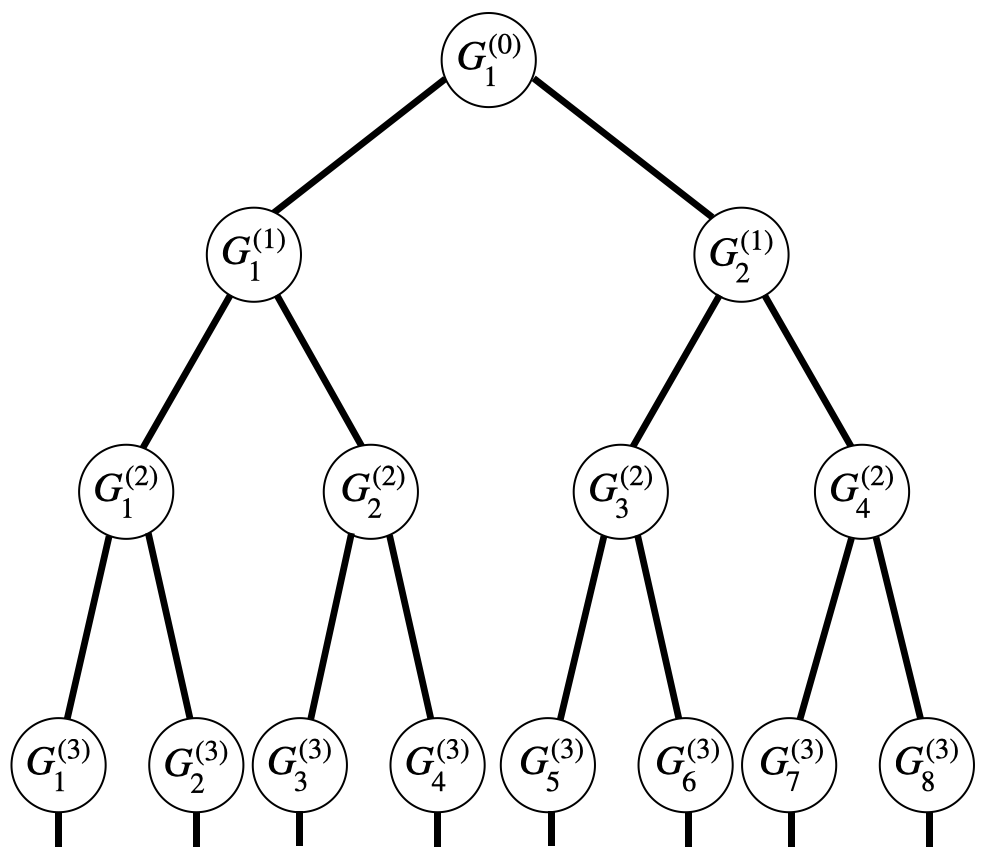}}}
    \caption{Tensor diagram of the hierarchical tensor-network representing an 8-dimensional density $p$.}\label{fig:h_tt}
\end{figure} 

At this point there are two issues that need to be addressed in order to use \eqref{eq:high d cdse} to determine the tensor cores $\{G^{(l)}_k\}_{k,l}$. The first issue is on how to obtain the range $p^{(l)}_{k}$ at each level. The second issue is that while one can determine the cores $\{G^{(l)}_k\}_{k,l}$ via \eqref{eq:high d cdse}, in practice it is challenging to estimate the exponential sized coefficient matrix $\{p^{(l)}_k\}_{k,l}$ (More specifically, it scales as $\vert \mathcal{I}^{(l)}_k\vert$, exponential with respect to $|\ccC_k^\bl|$) via only finite number of samples. To this end, we need to reduce the size of the system in \eqref{eq:high d cdse}.

We deal with these two issues as follows:
\begin{itemize}
\item \textbf{Obtaining $p^{(l)}_{k}$}:  Inspired by how one ``sketches'' the range of a matrix in randomized linear algebra literature \cite{oseledets2010tt,savostyanov2011fast}, we define 
\begin{equation}
p^{(l)}_k(\mathcal{I}^{(l)}_k,\tilde \gamma_k^{(l)}) = p(\mathcal{I}^{(l)}_k; \mathcal{J}^{(l)}_k) T_k^{(l)}(\mathcal{J}^{(l)}_k,\tilde \gamma_k^{(l)}), \ \ \tilde{\gamma}_k^{(l)} \in [\tilde{r}^{(l)}],\label{eq:computePlk}
\end{equation}
which has size $\vert \mathcal{I}^{(l)}_k \vert \times \tilde r^{(l)}$ and sketches the range of matrix $p(\mathcal{I}^{(l)}_k; \mathcal{J}^{(l)}_k)$ via multiplying it with sketch function $T_k^{(l)}(\mathcal{J}^{(l)}_k,\tilde \gamma_k^{(l)})$. It is desirable if the sketch function can capture the range: $\text{Range}(p^{(l)}_k(\mathcal{I}^{(l)}_k,\tilde \gamma_k^{(l)}) )= \text{Range}(p(\mathcal{I}^{(l)}_k; \mathcal{J}^{(l)}_k) T_k^{(l)}(\mathcal{J}^{(l)}_k,\tilde \gamma_k^{(l)}))=\text{Range}(p(\mathcal{I}^{(l)}_k; \mathcal{J}^{(l)}_k))$. The choice of sketch function will be discussed in Section~\ref{section:sketching}. 

\item \textbf{Reducing the number of equations}: 
Note that the system of equations in \eqref{eq:high d cdse} is over-determined, as each $G_k^{(l)}(:,:,\tilde \gamma_k^{(l)})$ is of size at most $\tilde r^{(l+1)} \times \tilde r^{(l+1)}$. Therefore in principle, one can reduce the rows and columns of each equation in \eqref{eq:high d cdse} such that each equation involves only  $\tilde r^{(l+1)} \times \tilde r^{(l+1)}$ coefficient matrices. This can be done by applying sketch function yet another time. In particular, we apply $S_{2k-1}^{(l)}(\mathcal{I}^{(l)}_{2k-1}, \tilde \nu_{2k-1}^{(l)})$ and $S_{2k}^{(l)}(\mathcal{I}^{(l)}_{2k}, \tilde \nu_{2k}^{(l)})$, $\tilde \nu_{2k-1}^{(l)}, \tilde \nu_{2k}^{(l)}\in [\tilde r^{(l)}] $ to both sides of the equations in \eqref{eq:high d cdse} which results in the following equations:
\begin{equation}
    A^{(l)}_{2k-1}(\tilde \nu_{2k-1}^{(l)},:)G_{k}^{(l-1)}(:,:,\tilde \gamma_{k}^{(l-1)})A^{(l)}_{2k}(\tilde{\nu}_{2k}^{(l)},:)^T= B_k^{(l-1)}(\tilde \nu^{(l)}_{2k-1},\tilde \nu^{(l)}_{2k},\tilde \gamma_k^{(l-1)}), \tilde \gamma_k^{(l-1)} \in [\tilde{r}^{(l-1)}],\label{eq:cde_sketch2}
\end{equation}
for $k \in [2^{l-1}], l \in [L]$ where
\begin{equation}
A^\bl_k(\tilde \nu_k^\bl,\tilde \gamma_k^\bl) = S_k^{(l)}(\ccI_k^\bl,\tilde \nu_k^\bl)^T p_k^\bl( \ccI_k^\bl,\tilde \gamma^\bl_k)=S_k^{(l)}(\ccI_k^\bl,\tilde \nu_k^\bl)^Tp(\mathcal{I}^{(l)}_k; \mathcal{J}^{(l)}_k) T_k^{(l)}(\mathcal{J}^{(l)}_k,\tilde \gamma_k^{(l)}), \label{eq:formulaA}
\end{equation}
and
\begin{equation}
    B_k^{(l)}(\tilde \nu^{(l+1)}_{2k-1},\tilde \nu^{(l+1)}_{2k},\tilde \gamma_k^\bl)=S^{(l+1)}_{2k-1}(\mathcal{I}^{(l+1)}_{2k-1},\tilde \nu_{2k-1}^{(l+1)})^T  p_{k}^{(l)}(\mathcal{I}^{(l+1)}_{2k-1},\mathcal{I}^{(l+1)}_{2k}, \tilde \gamma_{k}^{(l)}) S^{(l+1)}_{2k}(\mathcal{I}^{(l+1)}_{2k},\tilde \nu_{2k}^{(l+1)}). \label{eq:formulaB}
\end{equation}
At this point, one can solve for $G^{(l-1)}_k(:,:,\tilde \gamma_k^{(l-1)})$ in \eqref{eq:cde_sketch2} via pseudo-inverse 
\begin{align}\label{eq:solutionG_sketch}
&G^{(l-1)}_k(:,:,\tilde \gamma_k^{(l-1)})=({A^{(l)}_{2k-1}})^\dagger B_k^{(l-1)}(:,:,\tilde \gamma_k^{(l-1)})({A^{(l)T}_{2k}})^\dagger, k \in [2^{l-1}], l \in [L],  \nonumber \\
&G_k^{(L)}(\ccI_k^{(L)},\tilde \gamma_k^{(L)})=p_k^{(L)}(\ccI_k^{(L)},\tilde \gamma_k^{(L)}), k \in [2^{L}].
\end{align}
The equations with reduced size could be shown via the tensor diagram in Figure~\ref{fig:ceq_sktech}:
\begin{figure}[!htb]
    \centering   \subfloat{{\includegraphics[width=1.0\textwidth]{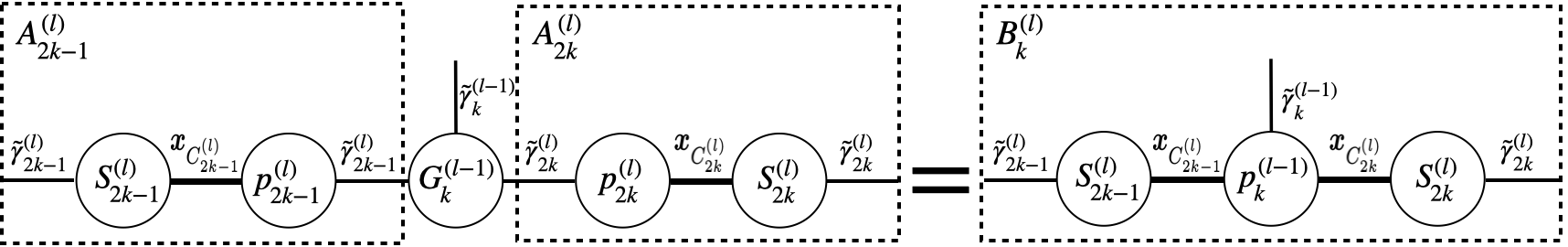}}}
    \caption{Tensor diagram of reduced core defining equations \eqref{eq:cde_sketch2}.}\label{fig:ceq_sktech}
\end{figure} 
\end{itemize}
The following theorem guarantees that the solution \eqref{eq:solutionG_sketch} to the reduced system of equations gives a tensor-network representation of $p$.
\begin{theorem}\label{thm:sketchrange}

If $\text{Range}\left(p(\mathcal{I}^{(l)}_k; \mathcal{J}^{(l)}_k)T_k^\bl\right)=\text{Range}\left(p(\mathcal{I}^{(l)}_k; \mathcal{J}^{(l)}_k)\right)$ and $\text{Range}\left(p( \mathcal{I}^{(l)}_k;\mathcal{J}^{(l)}_k)^TS_k^{(l)}\right)=\text{Range}\left(p(\mathcal{I}^{(l)}_k;\mathcal{J}^{(l)}_k)^T\right)$ for $k\in [2^l], l \in [L]$, then $\{\{G_k^\bl\}_{k=1}^{2^l}\}_{l=0}^{L}$ defined in \eqref{eq:solutionG_sketch} gives a hierarchical tensor-network representation of $p$.
\end{theorem}
\begin{proof}
We first show that $\{G^{(l)}_k\}$ in \eqref{eq:solutionG_sketch} gives a solution for \eqref{eq:high d cdse}, with the definition of $p_k^\bl$ in \eqref{eq:computePlk}:  
$p^{(l)}_k(\mathcal{I}^{(l)}_k,\tilde \gamma_k^{(l)}) = p(\mathcal{I}^{(l)}_k; \mathcal{J}^{(l)}_k) T_k^{(l)}(\mathcal{J}^{(l)}_k,\tilde \gamma_k^{(l)})$. This could help to establish  \eqref{eq:solutionG_sketch} as a tensor-network representation of $p$ via Theorem~\ref{thm:Grepre}. 

To begin with, we show $\text{Range}\left(A_k^{(l)T}\right)=\text{Range}\left(p_k^{(l)T}\right)$ for every $k,l$, indicating that the row spaces of $A_k^{(l)}$ and $p_k^{(l)}$ are the same. This is due to the fact that $\text{Range}\left(p_k^{(l)T}\right) \\ =\text{Range}\left(T_k^{(l)T} p(\mathcal{I}^{(l)}_k; \mathcal{J}^{(l)}_k)^T \right) = \text{Range}\left(T_k^{(l)T} p(\mathcal{I}^{(l)}_k; \mathcal{J}^{(l)}_k)^T S_k^\bl\right) = \text{Range}\left(A_k^{(l)T}\right)$. The first and last equalities are based on the definition of $p_k^\bl$ and $A_k^\bl$, and second equality is due to the 
assumption of the statement where $\text{Range}\left(p( \mathcal{I}^{(l)}_k;\mathcal{J}^{(l)}_k)^TS_k^{(l)}\right)=\text{Range}\left(p(\mathcal{I}^{(l)}_k;\mathcal{J}^{(l)}_k)^T\right)$ and applying $T_k^{(l)T}$ to the row space does not change the range. Thus $G_k^{(l-1)}$ defined in \eqref{eq:solutionG_sketch} satisfies equality
${p^{(l)}_{2k-1}} G^{(l-1)}_k(:,:,\tilde \gamma_k^{(l-1)}){p^{(l)T}_{2k}} = p_k^{(l-1)}(:,:,\tilde \gamma_k^{(l-1)})$ for $\tilde \gamma_k^{(l-1)}\in [\tilde{r}^{(l-1)}]$.

We have shown that $\{G_k^\bl\}$ in \eqref{eq:solutionG_sketch} is a solution to \eqref{eq:high d cdse} when $\{p^{(l)}_k\}$ is defined in \eqref{eq:computePlk}. The proof can be concluded if we further show $\{p^{(l)}_k\}$  in \eqref{eq:computePlk} satisfies Assumption~\ref{assumption:range}, hence Theorem~\ref{thm:Grepre} can be directly applied to show $\{G_k^\bl\}$ gives a representation of $p$. For a specific choice of $k$ and $l$,  $\text{Range}\left(p_{2k-1}^\bl\right)=\text{Range}\left(p(\mathcal{I}^{(l)}_{2k-1}; \mathcal{J}^{(l)}_{2k-1})T_{2k-1}^\bl\right)= \text{Range}\left(p(\mathcal{I}^{(l)}_{2k-1}; \mathcal{J}^{(l)}_{2k-1})\right)= \\ \text{Range}\left(p(\mathcal{I}^{(l)}_{2k-1}; \mathcal{I}^{(l)}_{2k},\mathcal{J}^{(l-1)}_{k})\right)\supset \text{Range}\left(p(\mathcal{I}^{(l)}_{2k-1}; \mathcal{I}^{(l)}_{2k},\tilde{\gamma}_k^{(l-1)})\right)$ where  the first equality is by definition \eqref{eq:computePlk}, second equality is due to the assumption of the theorem, the third equality is due to the fact that $\mathcal{J}^{(l)}_{2k-1}= \mathcal{I}^{(l)}_{2k}\times \mathcal{J}^{(l-1)}_{k}$  and the first inclusion holds since $p(\mathcal{I}^{(l)}_{2k-1}; \mathcal{I}^{(l)}_{2k},\tilde{\gamma}_k^{(l-1)})$ is formed from applying the sketch function $T_k^{(l-1)}$ to the columns of  $p(\mathcal{I}^{(l)}_{2k-1}; \mathcal{I}^{(l)}_{2k},\mathcal{J}^{(l-1)}_{k})$. Similarly, $\text{Range}(p_{2k}^\bl)\supset \text{Range}\left(p(\mathcal{I}^{(l)}_{2k}; \mathcal{I}^{(l)}_{2k-1},\tilde{\gamma}_k^{(l-1)})\right)$, thus Assumption~\ref{assumption:range} is satisfied. \end{proof}

Now, $\tilde r^{(l)},l=1,\ldots,L$ depends on the number of sketch functions chosen. As we tend to choose a large set of sketch functions in order to capture the range of $p(\mathcal{I}^{(l)}_k; \mathcal{J}^{(l)}_k)$ and ${p^{(l)T}_k}$  properly, this gives rise to three issues: (1) from a numerical standpoint, one may run into stability issue when  taking the pseudo-inverse of a nearly low rank matrix. If $p(\mathcal{I}^{(l)}_{k};\mathcal{J}^{(l)}_{k})$ is nearly rank $r^{(l)}\leq \tilde r^{(l)}$, then one would expect  $A^{(l)}_k(\tilde \nu_k^\bl,\tilde \gamma_k^\bl) = S_k^{(l)}(\ccI_k^\bl,\tilde \nu_k^\bl)^T p(\mathcal{I}^{(l)}_{k};\mathcal{J}^{(l)}_{k}) T_k^{(l)}(\ccJ_k^\bl,\tilde \gamma_k^\bl)$ to be also nearly rank $r^{(l)}\leq \tilde r^{(l)}$. Taking the pseudo-inverse of $A^{(l)}_k$ can cause instability. (2) From a statistical standpoint, it is challenging to estimate $\{A^{(l)}_k\}_{k,l}$ and $\{B_k^{(l)}\}_{k,l}$, the moments of $p$, from a fixed number of samples in $\hat p$ when $\tilde r^{(l)}$ is large. (3) From a computational viewpoint, large $\tilde r^{(l)}$ causes large $G^{(l)}_k$, leading to high complexity in storing and deploying the hierarchical tensor network. In order to solve these issues caused by over-sketching, we introduce a trimming strategy in Section~\ref{sec:trimming}.

\subsection{Trimming}\label{sec:trimming}
In practice, we often apply a low-rank truncation to each $A^{(l)}_k$. Let $A^{(l)}_k \rightarrow \mathcal{P}_{r^{(l)}}(A^{(l)}_k) $, where $\mathcal{P}_{r^{(l)}}(\cdot)$ provides the best rank-$r^{(l)}$ approximation to a matrix in terms of the spectral norm. Then we could solve
\begin{equation}\label{eq:projection}
    \mathcal{P}_{r^{(l)}}(A^{(l)}_{2k-1})G_{k}^{(l-1)}(:,:,\tilde \gamma_{k}^{(l-1)})\mathcal{P}_{r^{(l)}}(A^{(l)}_{2k})^T = B_k^{(l-1)}(:,:,\tilde \gamma_k^{(l-1)}), \tilde \gamma_k^{(l-1)} \in [\tilde{r}^{(l-1)}], k \in [2^{l-1}], l \in [L]
\end{equation}
instead of \eqref{eq:cde_sketch2}. The following proposition shows why this could be feasible:
\begin{prop}\label{prop:trim_rankr}
Under the assumption of Theorem~\ref{thm:sketchrange}, if $\text{rank}\left(p(\mathcal{I}^{(l)}_k; \mathcal{J}^{(l)}_k)\right)=r^\bl \leq \tilde{r}^\bl$ for $k\in [2^l], l\in [L]$, then $\mathcal{P}_{r^{(l)}}(A^{(l)}_{k})=A^{(l)}_{k}$. 
\end{prop}
\begin{proof}
For every $k\in [2^l], l\in [L]$, since $\text{Range}\left(p(\mathcal{I}^{(l)}_k; \mathcal{J}^{(l)}_k)T_k^\bl\right)=\text{Range}\left(p(\mathcal{I}^{(l)}_k; \mathcal{J}^{(l)}_k)\right)$ and $\text{rank}\left(p(\mathcal{I}^{(l)}_k; \mathcal{J}^{(l)}_k)\right)=r^\bl \leq \tilde{r}^\bl$, then $\text{rank}\left(T_k^\bl\right)=r^\bl$. Similarly, $\text{rank}\left(S^{(l)}_k\right)=r^\bl$ based on
$\text{Range}\left(p( \mathcal{I}^{(l)}_k;\mathcal{J}^{(l)}_k)^TS_k^{(l)}\right)=\text{Range}\left(p(\mathcal{I}^{(l)}_k;\mathcal{J}^{(l)}_k)^T\right)$. Furthermore, 
$A^\bl_k =S_k^{(l)}(\ccI_k^\bl,:)^Tp(\mathcal{I}^{(l)}_k; \mathcal{J}^{(l)}_k)\\ T_k^{(l)}(\mathcal{J}^{(l)}_k,:) \in \bbRm{\tilde{r}^\bl}{\tilde{r}^\bl}$ in \eqref{eq:formulaA} is also of rank $r^\bl$. Therefore, $\mathcal{P}_{r^{(l)}}(A^{(l)}_{k})=A^{(l)}_{k}$ due to the fact that $\mathcal{P}_{r^{(l)}}(\cdot)$ provides the best rank-$r^{(l)}$ approximation.   
\end{proof}
For $A^{(l)}_k$ with exactly rank $r^{(l)}$, the projection $\mathcal{P}_{r^{(l)}}(\cdot)$ in \eqref{eq:projection} seems vacuous following Proposition~\ref{prop:trim_rankr}. However, this becomes crucial when $A^{(l)}_k$ is estimated from empirical moments, whose rank is higher than $r^\bl$, as mentioned in the end of last subsection. Now the solution becomes
\begin{equation}
    G_{k}^{(l-1)}(:,:,\tilde \gamma_{k}^{(l-1)}) = \left(\mathcal{P}_{r^{(l)}}(A^{(l)}_{2k-1})\right)^\dagger  B_k^{(l-1)}(:,:,\tilde \gamma_k^{(l-1)})\left(\mathcal{P}_{r^{(l)}}(A^{(l)}_{2k})^T\right)^{\dagger}, \tilde \gamma_k^{(l-1)} \in [\tilde{r}^{(l-1)}], \label{eq:solveG}
\end{equation}
for $k \in [2^{l-1}], l\in [L]$. Let the right singular matrices of $\mathcal{P}_{r^{(l)}}(A^{(l)}_{2k-1})$ and $\mathcal{P}_{r^{(l)}}(A^{(l)}_{2k})$ be $V^{(l)}_{2k-1},V^{(l)}_{2k}\in  \bbRm{\tilde{r}^{(l)}}{r^{(l)}}$. Since the range of $G^{(l-1)}_k(:,:,\tilde \gamma_k^{(l-1)})$ is given by $V^{(l)}_{2k-1},V^{(l)}_{2k}$, one can have a reduced size tensor-network with rank $\{r^{(l)}\}_l$ via defining
\begin{equation}
    C^{(l-1)}_k(:,:, \gamma_k^{(l-1)})  =  \sum_{\tilde \gamma_k^{(l-1)} \in [\tilde{r}^{(l-1)}]} \left( {V^{(l)T}_{2k-1}} G_{k}^{(l-1)}(:,:,\tilde \gamma_k^{(l-1)})  {V^{(l)}_{2k}}\right)  V^{(l-1)}_{k}(\tilde \gamma_k^{(l-1)},\gamma_k^{(l-1)}), \gamma_k^{(l-1)} \in [r^{(l-1)}],\label{eq:computeC}
\end{equation}
for $k \in [2^{l-1}], l\in [L]$. Besides, we have a similar definition of $C_k^{(L)}$ in the last level:
\begin{equation}
C_k^{(L)}(\mathcal{I}^{(l)}_k, \gamma_k^{(l)})=G^{(L)}_k(\mathcal{I}^{(L)}_k,:)V_k^{(L)}(:, \gamma_k^{(L)}),k\in [2^L].\label{eq:computeCkL}    
\end{equation}
The set of nodes $\{\{C_k^\bl\}_{k=1}^{2^l}\}_{l=0}^L$ gives rise to a hierarchical tensor-network in Figure~\ref{fig:trimming},  via inserting $V^{(l)}_{2k-1}V^{(l)T}_{2k-1},V^{(l)}_{2k}V^{(l)T}_{2k}$ for the corresponding $G_k^{(l-1)}$ in the hierarchical tensor-network and regrouping the components, as shown in Figure~\ref{fig:trimming}(A). To highlight that the smaller size of $C_k^\bl$ compared to $G_k^\bl$, we represent legs of $C_k^\bl$ with thin lines and legs of $G_k^\bl$ with bold lines.
In the following sections, we use this trimmed hierarchical tensor-network in Figure~\ref{fig:trimming}(B) as a representation of probability $p$.

\begin{figure}[!htb]
    \centering
    \subfloat[Insert $V_k^\bl V_k^{(l)T}$ and regroup the components in the hierarchical tensor-network.]{{\includegraphics[width=0.48\textwidth]{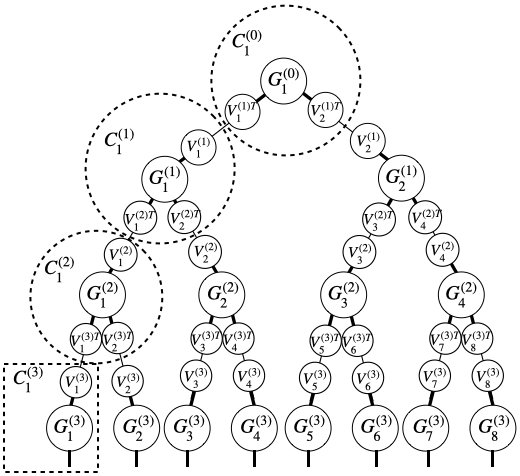}}}
    \quad 
    \subfloat[Trimmed hierarchical tensor-network from simplifying the representation with new variables. ]{{\includegraphics[width=0.48\textwidth]{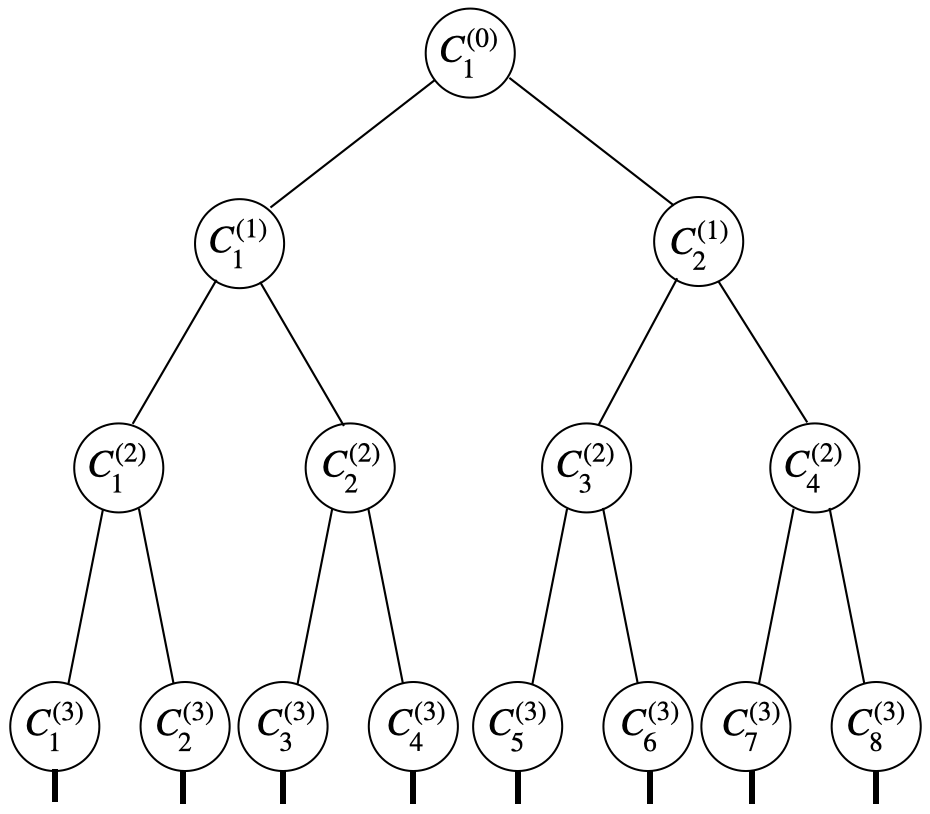}}}
    \caption{Tensor diagram of trimmed hierarchical tensor-network. }
    \label{fig:trimming}
\end{figure}

\subsection{Choice of sketch function}\label{section:sketching}

In this subsection, we introduce the choice of sketch function $\{S^{(l)}_k\}_{k,l}$ and $\{T^{(l)}_k\}_{k,l}$. There are two criteria for such a selection. (1) Since each $A_k^\bl$ is a collection of moments, one needs to be able to efficiently estimate $A_k^\bl$ from the empirical distribution $\hat p$. (2) $T_k^\bl$ needs to capture the range of $p(\mathcal{I}_k^{(l)};\mathcal{J}_k^{(l)})$ and furthermore, $S^{(l)}_k$ needs to capture the row space of $p(\mathcal{I}_k^{(l)};\mathcal{J}_k^{(l)})$, as the assumption in Theorem~\ref{thm:sketchrange}.

We deal with the choice of basis via using a \emph{cluster} basis, which is successfully implmented in  atomic cluster expansion \cite{drautz2019atomic,dusson2022atomic}. Starting from a single variable basis set $\{\phi_i\}^n_{i=1}$, we construct the cluster set of one variable:
\begin{equation}
   \mathbf{B}^{1,d}:=\bigcup_{k_1=1}^d \left(\bigcup_{i_{k_1}=1}^n\{\phi_{i_{k_1}}(x_{k_1})\}\right),
\end{equation}
and cluster set of two variables:
\begin{equation}
    \mathbf{B}^{2,d}:=\bigcup_{(k_1,k_2)\in {[d]\choose 2}}\left(\bigcup_{i_{k_1},i_{k_2}=1}^n\{\phi_{i_{k_1}}(x_{k_1})\phi_{i_{k_2}}(x_{k_2})\}\right), 
\end{equation}
and cluster set of $t$ variables in general:
\begin{equation}
\mathbf{B}^{t,d}:=\bigcup_{(k_1,\ldots,k_t)\in {[d] \choose t}}\left(\bigcup_{i_{k_1},\cdots,i_{k_t}=1}^n\{\phi_{i_{k_1}}(x_{k_1})\cdots \phi_{i_{k_t}}(x_{k_t})\}\right).
\end{equation} 
Commonly, $\{\phi_i\}^n_{i=1}$ are chosen as a set of orthonormal basis functions for convenience. 

The way that we apply the sketching function is straightforward. We define 

\begin{equation}
S^{(l)}_{k}(\cdot,\tilde \nu^{(l)}_{k})\in  \mathbf{B}^{t,\vert \mathcal{C}^{(l)}_{k}\vert}, \ \ T^{(l)}_k(\cdot,\tilde \gamma_k^{(l)})\in \mathbf{B}^{t,d-\vert \mathcal{C}^{(l)}_{k}\vert}, 
\end{equation}
where clusters $\mathcal{C}^{(l)}_k=(k-1)(\frac{2^L}{2^l})+1:k(\frac{2^L}{2^l})$ for $k\in[2^l],l\in[L]$, as defined previously in Section~\ref{sec:main idea}. Therefore, to obtain $A^{(l)}_{k}$ in \eqref{eq:formulaA},
we evaluate  $S^{(l)}_{k}(\cdot,\tilde \nu^{(l)}_{k}):\vx_{\mathcal{C}^{(l)}_{k}}\rightarrow \mathbb{R}$ for all possible choices of $\vx_{\mathcal{C}^{(l)}_{k}}$ to obtain the sketch matrix $S^{(l)}_{k}(\ccI_k^{(l)},\tilde \nu^{(l)}_{k})$, and evaluate $T^{(l)}_{k}(\cdot,\tilde \gamma^{(l)}_k):\vx_{[d]\backslash\mathcal{C}^{(l)}_{k}}\rightarrow \mathbb{R}$ for all possible choices of $\vx_{[d]\backslash\mathcal{C}^{(l)}_{k}}$ to obtain the sketch matrix $T^{(l)}_{k}(\ccJ_k^{(l)},\tilde \gamma^{(l)}_k)$. In practice, when we estimate $A^{(l)}_{k}$ from empirical distribution $\hat p$ we do not need to consider all possible values of $\vx_{\mathcal{C}^{(l)}_{k}},\vx_{[d]\backslash\mathcal{C}^{(l)}_{k}}$ when forming $A^{(l)}_k$, but only those in the support of $\hat p$. Therefore, we only need to evaluate the sketch functions $N$ times. We again use $S^{(l)}_{2k-1}(\ccI_{2k-1}^{(l)},\tilde \nu^{(l)}_{2k-1}),S^{(l)}_{2k}(\ccI_{2k}^{(l)},\tilde \nu^{(l)}_{2k})$ and $T^{(l-1)}_{k}(\ccJ_k^{(l-1)},\tilde \gamma^{(l-1)}_k)$ together with $p(\mathcal{I}^{(l)}_{2k-1},\mathcal{I}^{(l)}_{2k},\mathcal{J}^{(l-1)}_{k} )$ to obtain $B^{(l-1)}_k$ in \eqref{eq:formulaB}.

Under this construction, sketch index set $\{\tilde{\nu}^{(l)}_{k}\}$ has cardinality ${\vert \mathcal{C}^{(l)}_{k}\vert \choose t}n^t$, significantly larger than true rank $r^\bl$ of $p(\ccI_k^\bl;\ccJ_k^\bl)$, especially in the case when cluster size $\vert \mathcal{C}^{(l)}_{k}\vert$ is large. As a result, the sketch may be substantially oversized, which can in turn create difficulties in the subsequent trimming step. \par 
To solve above issue, we develop randomized cluster basis following randomized singular value decomposition technique \cite{halko2011finding}. For our problem, we assume that $\text{rank}\left(p(\ccI_k^\bl;\ccJ_k^\bl)\right)=r^\bl$ is small, so a sketch with only a modest number of columns is sufficient to capture the relevant column space. Let $\tilde{S}^{(l)}_{k}(\cdot,\cdot)\in  \mathbf{B}^{t,\vert \mathcal{C}^{(l)}_{k}\vert}$ denote the original cluster basis. We define the randomized basis
\begin{equation}
    S^{(l)}_{k}(\vx_{\mathcal{C}^{(l)}_{k}},\tilde \nu^{(l)}_{k})=\tilde{S}^{(l)}_{k}(\vx_{\mathcal{C}^{(l)}_{k}},:)W(:,\tilde \nu^{(l)}_{k}), \tilde \nu^{(l)}_{k}\in [\tilde{r}^\bl], 
\end{equation}
where $W \in \bbRm{{\vert \mathcal{C}^{(l)}_{k}\vert \choose t}n^t}{\tilde{r}^\bl}$ is a random matrix and we assume that it has orthonormal columns for convenience: $\langle W(:,i), W(:,j) \rangle = \delta_{i,j}$. Here we could choose $\tilde{r}^\bl$, a dimension-independent constant, slightly larger than true rank $r^\bl$ of $p(\ccI_k^\bl;\ccJ_k^\bl)$, to avoid oversketching issue generated from original cluster basis. Besides, we assume cluster basis matrix $\tilde{S}_k^\bl$ have orthonormal columns as well. In the end, randomized cluster basis $S_k^\bl$ have orthonormal columns: $\langle S_k^\bl(:,i), S_k^\bl(:,j) \rangle = \langle \tilde{S}_k^\bl(:,:)W(:,i), \tilde{S}_k^\bl(:,:)W(:,j) \rangle = \langle W(:,i), W(:,j) \rangle = \delta_{i,j}$. We repeat this procedure for $T^{(l)}_k(\cdot,\tilde \gamma_k^{(l)})$ to obtain an orthonormal randomized basis in the same way. \par

Another improvement is to construct the random coefficient $W$ using a tensor-network parameterization. Specifically, for each column $W(:,i)$, we model it either as a rank-one tensor product or as a tensor-train representation with rank $\tilde{r}$. This structured construction can further reduce computational cost and improve overall efficiency.

\subsection{Algorithm and complexity}

In this section, we summarize the steps for obtaining $\{C_k^\bl\}_{k,l}$, the tensor-network representation of a density $p$, in Algorithm~\ref{alg:downsweep}.

\begin{algorithm}
\caption{Algorithm for hierarchical tensor-network estimation (algorithm denoted as $\mathcal{A}(\cdot)$).}
\label{alg:downsweep}
    \begin{algorithmic}[1]
        \Require Input $p:[n_1]\times \cdots \times [n_d]\rightarrow \mathbb{R}$. Sets of sketch functions $\{S_k^\bl\}_{k,l}$ and $\{T_k^\bl\}_{k,l}$. Rank at each level $\{r^\bl\}_l$.  

        \For {$l=1,\cdots,L$}
            \For {$k=1,\cdots,2^{l-1}$}
            \State Obtain $p_{i}^\bl$ in \eqref{eq:computePlk} with $T_{i}^{(l)}$ and $p$, $i=2k-1,2k$.
            \State Obtain $A_{i}^{(l)}$ in \eqref{eq:formulaA} with $S_{i}^{(l)}$ and $p_{i}^\bl$,  $i=2k-1,2k$.
            \State Obtain $B_k^{(l-1)}$ in \eqref{eq:formulaB} with $S_{2k-1}^{(l)}$, $S_{2k}^{(l)}$ and $p_k^{(l-1)}$.
            \State Trim $A_{2k-1}^\bl$ and $A_{2k}^\bl$ with rank $r^\bl$ and solve for $G_k^{(l-1)}$ in \eqref{eq:solveG} in terms of $C_k^{(l-1)}$  in \eqref{eq:computeC}.
            \EndFor
        \EndFor

        \State Obtain $C_k^{(L)}$ in \eqref{eq:computeCkL} for $k=1,\cdots,2^L$. 
        \State \textbf{Return} $d$-dimensional function represented in tensor-network  $\{\{C_k^\bl\}_{k=1}^{2^{l}}\}_{l=0}^{L}$ as in Figure~\ref{fig:trimming}(B).
        \par 

    \end{algorithmic}
\end{algorithm}
In practice, we apply this algorithm to empirical distribution $\hat{p}$, which is constructed from $N$ independent and identically distributed samples. So $\hat{p}$ has at most $N$ non-zero entries. We denote $\tilde{r}_{\texttt{max}}=\max_{l}\tilde{r}^{(l)}$ and $r_{\texttt{max}}=\max_{l}r^{(l)}$ for convenience where $\tilde{r}_{\texttt{max}}>r_{\texttt{max}}$. For convenience, we assume cluster sketching order $t=1$ in Section~\ref{section:sketching}. To calculate the complexity, we assume that evaluating a function on a single entry is $O(1)$. For each specific $k,l$, ($k\in [2^l], l\in [L]$)
\begin{enumerate}
    \item $S_i^{(l)},T_i^{(l)}$ can be computed in $O(\tilde{r}_{\texttt{max}}dN)$ time, $i=2k-1,2k$. 
    \item $p_i^{(l)}$ can be computed in $O(\tilde{r}_{\texttt{max}}N)$ time, $i=2k-1,2k$. 
    \item $A_i^{(l)}$ can be computed in $O(\tilde{r}_{\texttt{max}}^2N)$ time, $i=2k-1,2k$.
    \item $B_k^{(l-1)}$ can be computed in $O(\tilde{r}_{\texttt{max}}^3N)$ time. 
    \item $\mathcal{P}_{r^{(l)}}(A^{(l)}_i)$ can be computed in $O(\tilde{r}_{\texttt{max}}^2r_{\texttt{max}})$ time, $i=2k-1,2k$. 
    \item $G_k^{(l-1)}$ can be computed in $O(\tilde{r}_{\texttt{max}}^3r_{\texttt{max}}+\tilde{r}_{\texttt{max}}^2r_{\texttt{max}}^2)$ time. 
    \item $C_k^{(l-1)}$ can be computed in $O(\tilde{r}_{\texttt{max}}^3r_{\texttt{max}}+\tilde{r}_{\texttt{max}}^2r_{\texttt{max}}^2+\tilde{r}_{\texttt{max}}r_{\texttt{max}}^3)$ time. 
\end{enumerate}
The number of $k,l$ is at most $d\log_2{(d)}$ in total. Note that the computation of sketching function in the first step can be implemented in a parallel way. Therefore, the total complexity of this algorithm is 
$$O\left(d\log{(d)}\ (\tilde{r}_{\texttt{max}}^3N+  \tilde{r}_{\texttt{max}}^3r_{\texttt{max}})\right)$$ which depends linearly on $N$ and near linearly on $d$.

\section{Analysis of estimation error}\label{sec:error analysis}

In this section, we provide analysis regarding the estimation error. Denote $p^*$, $\hat{p}$, $\tilde{p}=\mathcal{A}(\hat p)$ as the ground truth distribution, empirical distribution in \eqref{eq:empirical_density} and the tensor-network obtained from Algorithm~\ref{alg:downsweep}. Recall that the total error is given by
\begin{equation}
    p^* - \mathcal{A}(\hat p) = p^* - \mathcal{A}(p^*) + \mathcal{A}(p^*) -  \mathcal{A}(\hat p),
\end{equation}
where $\mathcal{A}$ is the proposed Algorithm~\ref{alg:downsweep}. 

We apply the following blanket assumption throughout the section.
\begin{assumption}
We assume the ground truth distribution $p^*$ could be represented by a hierarchical tensor-network with rank $\{r^{(l)}\}^{L}_{l=1}$ and further, the assumption of Theorem~\ref{thm:sketchrange} is satisfied by $p=p^*$ with our choice of sketch functions. 
\end{assumption}
In this case, no approximation error is committed when applying Algorithm~\ref{alg:downsweep} to $p^*$, i.e. $\mathcal{A}(p^*) = p^*$. In practice, this may not be true, and we discuss such an error in Section~\ref{sec:numerical} via numerical experiments. With such an assumption, the total error could be simplified as $p^* - \mathcal{A}(\hat p) = p^*- \tilde p$, which is the estimation error caused by the sample variance in $\hat p$. Our goal is to prove the stability of $\|p^* - \tilde p\|_F$ in terms of $N$ and $d$.

We summarize the notations in this section:
\begin{enumerate}
    \item We denote $G_k^\bl\rightarrow G_k^{\bl *}$ if we input $p^*$ to $\mathcal{A}(\cdot)$, and $G_k^\bl\rightarrow {\hat G_k^\bl}$  if we input $\hat p$ to $\mathcal{A}(\cdot)$ where tensor core $G_k^\bl$ is defined in \eqref{eq:solveG}.
    \item We let $A_k^\bl\rightarrow {A_k^{\bl*}}$ if we input $p^*$ to $\mathcal{A}(\cdot)$, and $A_k^\bl\rightarrow {\hat A_k^\bl}$  if we input $\hat p$ to $\mathcal{A}(\cdot)$ where $A_k^\bl$ is defined in \eqref{eq:formulaA}.
    \item We let $B_k^\bl\rightarrow {B_k^{\bl*}}$ if we input $p^*$ to $\mathcal{A}(\cdot)$, and $B_k^\bl\rightarrow {\hat B_k^\bl}$  if we input $\hat p$ to $\mathcal{A}(\cdot)$ where $B_k^\bl$ is defined in \eqref{eq:formulaB}.
    \item We let $p_k^\bl\rightarrow {p_k^{\bl*}}$ when $p\rightarrow  p^*$  where $p_k^\bl$ is defined in     \eqref{eq:computePlk}. This encodes the ranges of various unfolding matrices of $p^*$. We also denote $\tilde{p}_k^\bl$ as estimated density at $k$-th node in $l$-th level, following the recursion level by level in \eqref{eq:recursion}:
    \begin{align}\label{eq:recursion}
    \tilde{p}_k^{(L)}(\ccI_k^{(L)},:)=\hat{p}(\ccI_k^{(L)},:)T_k^{(L)}, k\in[2^L], \nonumber\\
      \tilde{p}_k^{(l-1)} = \tilde{p}_{2k-1}^{(l)}\hat{G}_k^{(l-1)}\tilde{p}_{2k}^{(l)T}, k\in[2^{l-1}],l\in[L].
    \end{align}
    In the last level, $\tilde{p}_k^{(L)}(\ccI_k^{(L)},:)$ reflects the sketched range of the empirical distribution \\ $\hat{p}(\ccI_k^{(L)},\ccJ_k^{(L)})$, a good candidate of marginal density to recover the distribution based on this hierarchical structure. 

    \item We use $\ccP_{r^{(l)}}(A^{(l)*}_{k})$ and $\ccP_{r^{(l)}}(\hat{A}^{(l)}_{k})$ to represent the best rank-$r^{(l)}$ approximation matrix of $A^{(l)*}_{k} \in \bbRm{\tilde{r}^\bl}{\tilde{r}^\bl}$ and $\hat{A}^{(l)}_{k} \in \bbRm{\tilde{r}^\bl}{\tilde{r}^\bl}$, respectively. We recall notation $\tilde{r}_{\texttt{max}}=\max_{l}\tilde{r}^{(l)}$.
    \item We use $\epsilon_{A}$ to represent the maximum of the following Frobenius norm:$\epsilon_{A}=\max_{k,l}\|\hat{A}_k^\bl - A_k^{\bl*} \|_F$.
    \item We use $c_{A^\dagger}, c_{\hat{A}^\dagger}, c_S, c_{S^\dagger},c_p,c_{\tilde{p}}$ as the upper bounds of the following terms:
    \begin{align}\label{eq:constant def}    &\|\left(A_k^{\bl*}\right)^\dagger\|_2\leq c_{A^\dagger},\|\left(\ccP_{r^{(l)}}(\hat{A}^{(l)}_k)\right)^\dagger\|_2 \leq c_{\hat{A}^\dagger}, \nonumber\\ 
    &\|S_k^\bl\|_2\leq c_S, \|\left(S_k^\bl\right)^\dagger\|_2\leq c_{S^\dagger},\|p_k^{(l)*}\|_F \leq c_p, \|\tilde{p}_k^{(l)}\|_F \leq c_{\tilde{p}} , k\in[2^l],l\in [L].
     \end{align}
     We also use $c_p$ as an upper bound in the following: $c_p\geq \|p_1^{(0)*}\|_F=\|p^*\|_F$.
\end{enumerate}
Here we summarize our main results in the following theorem:
\begin{theorem}\label{thm:main}
Suppose Assumption~\ref{assumption:range} holds and sketches are orthogonal, i.e. for each $k,l$, $\langle S^\bl_k(:,i), S^\bl_k(:,j)\rangle =\delta_{ij}$, $\langle T^\bl_k(:,i), T^\bl_k(:,j)\rangle =\delta_{ij}$, $i,j\in [\tilde{r}^\bl]$. Then for any $0<\delta<1$, when $c^2_{A^\dagger}c_pc_{\tilde{p}}+c^2_{A^\dagger}c^2_p>1$, the following inequality holds with probability at least $1-\delta$:
\begin{equation}
\|\tilde{p}-p^*\|_F \leq 
\tilde{C}(c^2_{A^\dagger}c_pc_{\tilde{p}}+c^2_{A^\dagger}c^2_p)^{\log_2{d}}\frac{\log(\frac{2\tilde{r}_{\texttt{max}}d\log_2{d}}{\delta})}{\sqrt{N}}
\end{equation}
where $\tilde{C}=3\sqrt{\tilde{r}_{\texttt{max}}}\left(1+\frac{c^2_{\tilde{p}}(c^2_{\hat{A}^\dagger}   + 6c_{A^\dagger}c^2_{\hat{A}^\dagger}c_p + 6c^2_{A^\dagger}c_{\hat{A}^\dagger}c_p)}{c^2_{A^\dagger}c_pc_{\tilde{p}}+c^2_{A^\dagger}c^2_p-1} \right)$.
\end{theorem}
\begin{proof}
By assembling Lemma\ref{lemma:errorG} and Lemma\ref{lemma:error_p} in \ref{sec:mainproof}, we have 
\begin{equation}\label{eq:main_proof}
\|\tilde{p}-p^*\|_F \leq \kappa_{p^{(0)}}\epsilon_A=
(c^2_{A^\dagger}c_pc_{\tilde{p}}+c^2_{A^\dagger}c^2_p)^{\log_2{d}}\left(1+\frac{c^2_{\tilde{p}}\kappa_{G}}{c^2_{A^\dagger}c_pc_{\tilde{p}}+c^2_{A^\dagger}c^2_p-1} \right)\epsilon_A,
\end{equation}
where $\kappa_G=c^2_{\hat{A}^\dagger}   + 6c_{A^\dagger}c^2_{\hat{A}^\dagger}c_p + 6c^2_{A^\dagger}c_{\hat{A}^\dagger}c_p$. By plugging upper bound $\epsilon_A$ from Lemma \ref{lemma:epsilon_A} into \eqref{eq:main_proof}, we have the following inequality holds with probability at least $1-\delta$:

\begin{equation}
\|\tilde{p}-p^*\|_F \leq 
\tilde{C}(c^2_{A^\dagger}c_pc_{\tilde{p}}+c^2_{A^\dagger}c^2_p)^{\log_2{d}}\frac{\log(\frac{2\tilde{r}_{\texttt{max}}d\log_2{d}}{\delta})}{\sqrt{N}},
\end{equation}
where $\tilde{C}=3\sqrt{\tilde{r}_{\texttt{max}}}\left(1+\frac{c^2_{\tilde{p}}(c^2_{\hat{A}^\dagger}   + 6c_{A^\dagger}c^2_{\hat{A}^\dagger}c_p + 6c^2_{A^\dagger}c_{\hat{A}^\dagger}c_p)}{c^2_{A^\dagger}c_pc_{\tilde{p}}+c^2_{A^\dagger}c^2_p-1} \right)$.
\end{proof}

Inspired by \cite{qin2022error}, the theorem is aimed to bound the error in terms of the Frobenius norm. In the following parts, we would provide two remarks about this theorem. Next we show some preliminaries in  \ref{sec:Preliminaries} and present the components including Lemma~\ref{lemma:errorG}, Lemma~\ref{lemma:error_p}, and Lemma~\ref{lemma:epsilon_A} in \ref{sec:mainproof}. 
\begin{remark}
We remark that the assumption $c^2_{A^\dagger}c_pc_{\tilde{p}}+c^2_{A^\dagger}c^2_p>1$ can always be satisfied by a sufficient large choice of the constant $c_{A^\dagger}$.
\end{remark}

\begin{remark}
Furthermore, based on the fact that $p^*$ should be density, we could have bounds for $c_p$ to simplify the result. More specifically, the sum of the absolute value of entries of $p^*$ are $1$ since $p^*$ is a density. Thus $\|p_k^{(l)*}\|_F$ has the following upper bound: $\|p_k^{(l)*}\|_F \leq \|p^*T_k^\bl\|_F\leq \|p^*\|_F\|T_k^\bl\|_2\leq 1$ for every $k,l$ and therefore $c_p\leq 1$. 
 Besides, if $c_{\tilde{p}}\leq 1$ also holds, the  simplified inequality holds with probability at least $1-\delta$:
\begin{equation}
\|\tilde{p}-p^*\|_F \leq 
\tilde{C}'d(c_{A^\dagger})^{2\log_2{d}}\frac{\log(\frac{2\tilde{r}_{\texttt{max}}d\log_2{d}}{\delta})}{\sqrt{N}},
\end{equation}
where $\tilde{C}'=3\sqrt{\tilde{r}_{\texttt{max}}}\left(1+\frac{c^2_{\hat{A}^\dagger}   + 6c_{A^\dagger}c^2_{\hat{A}^\dagger} + 6c^2_{A^\dagger}c_{\hat{A}^\dagger}}{2c^2_{A^\dagger}-1} \right)$.\par 
If $c_p\leq 1 \leq c_{\tilde{p}}$, then the following inequality holds with probability at least $1-\delta$:
\begin{equation}
\|\tilde{p}-p^*\|_F \leq 
\tilde{C}''d(c_{A^\dagger}c_{\tilde{p}})^{2\log_2{d}}\frac{\log(\frac{2\tilde{r}_{\texttt{max}}d\log_2{d}}{\delta})}{\sqrt{N}},
\end{equation}
where $\tilde{C}''=3\sqrt{\tilde{r}_{\texttt{max}}}\left(1+\frac{c_{\tilde{p}}^2(c^2_{\hat{A}^\dagger}   + 6c_{A^\dagger}c^2_{\hat{A}^\dagger} + 6c^2_{A^\dagger}c_{\hat{A}^\dagger})}{2c^2_{A^\dagger}c_{\tilde{p}}^2-1} \right)$.\par

\end{remark}

\section{Numerical results}\label{sec:numerical}
In this section, we demonstrate the success of the algorithm on various one-dimensional and two-dimensional Ising models. \footnote{The implementation of hierarchical tensor sketching can be found at \url{https://github.com/IvanPeng0414/Hierarchical-Tensor-Sketching}.}  We recall notation $p^*$, $\hat{p}$, $\tilde{p}=\mathcal{A}(\hat p)$ as the ground-truth distribution, the empirical distribution and the tensor-network represented distribution, respectively. 
We can compute the relative error with respect to $p^*$:
$$\epsilon_p=\frac{\|\tilde{p}-p^*\|_F}{\|p^*\|_F}$$

For the choice of sketch function in Section~\ref{section:sketching}, we are concerned with Ising-like models where $n=2$. 

\subsection{One-dimensional Ising model}\par 
\subsubsection{Ferromagnetic Ising model}
We consider the following generalization of the one-dimensional Ising model. Define $p^*:\{-1,1\}^d\rightarrow \mathbb{R}$ by
$$p^*(x_1,\cdots,x_d) \propto \exp(-\beta \sum_{i,j=1}^d J_{ij}x_ix_j)$$
where $\beta>0$ reflects the inverse of temperature in this model and the interaction coefficient $J_{ij}$ is given by
$$J_{ij}=
\begin{cases}
-1/2, & |i-j|=1 \\
-1/6, & |i-j|=2 \\
0, & \texttt{otherwise}
\end{cases}$$
giving interactions between both nearest and next-nearest neighbors. We first investigate the error with respect to the number of samples ($N$) and show the resulting error in Figure~\ref{fig:error_1d_ferro} for 
various temperatures. A typical distribution is visualized in Figure~\ref{fig:empirical_1d_antiferro}(A). The distribution concentrates on two states, where one of them has entirely positive signs and the other has all negative signs.

\begin{figure}[!htb]
    \centering  {\includegraphics[width=0.6\textwidth]{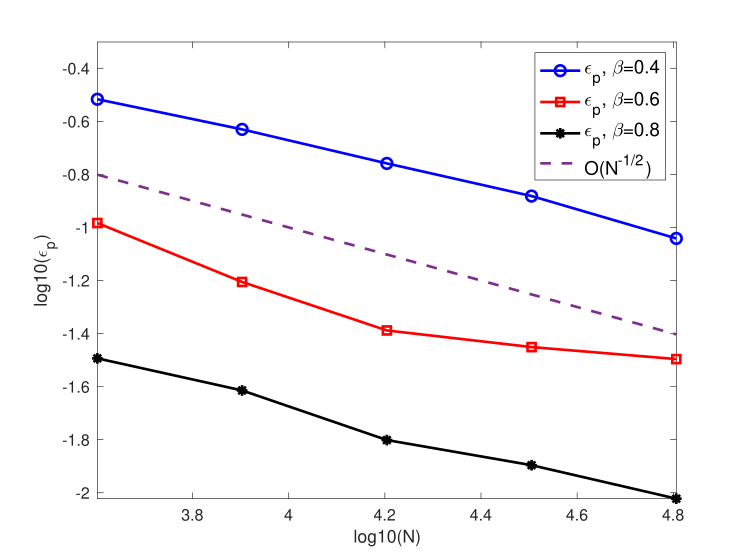}}
    \caption{Error for next neighbor one-dimensional Ising model with $d=16$. Blue, red, and black curves represent cases with $\beta=0.4$, $\beta=0.6$, and $\beta=0.8$, respectively. Dashed curve reflects a reference curve $O(N^{-1/2})$. }\label{fig:error_1d_ferro}
\end{figure}

Based on the results obtained, it can be observed that the error decreases as  $\frac{1}{\sqrt{N}}$, which satisfies the Monte-Carlo rate. Additionally, it is noticeable that the error is smaller in the case of lower temperature examples. This can be explained by the fact that the distribution is more concentrated at lower temperatures, giving less variance. Additionally, we plot the error v.s. $d$ in Figure~\ref{fig:error_changed_ferro} with temperature $\beta=0.6$ and  $N=64000$ empirical samples and observe a mild growth with the dimensionality $d$.

\par

\begin{figure}[!htb]
    \centering
    {\includegraphics[width=0.6\textwidth]{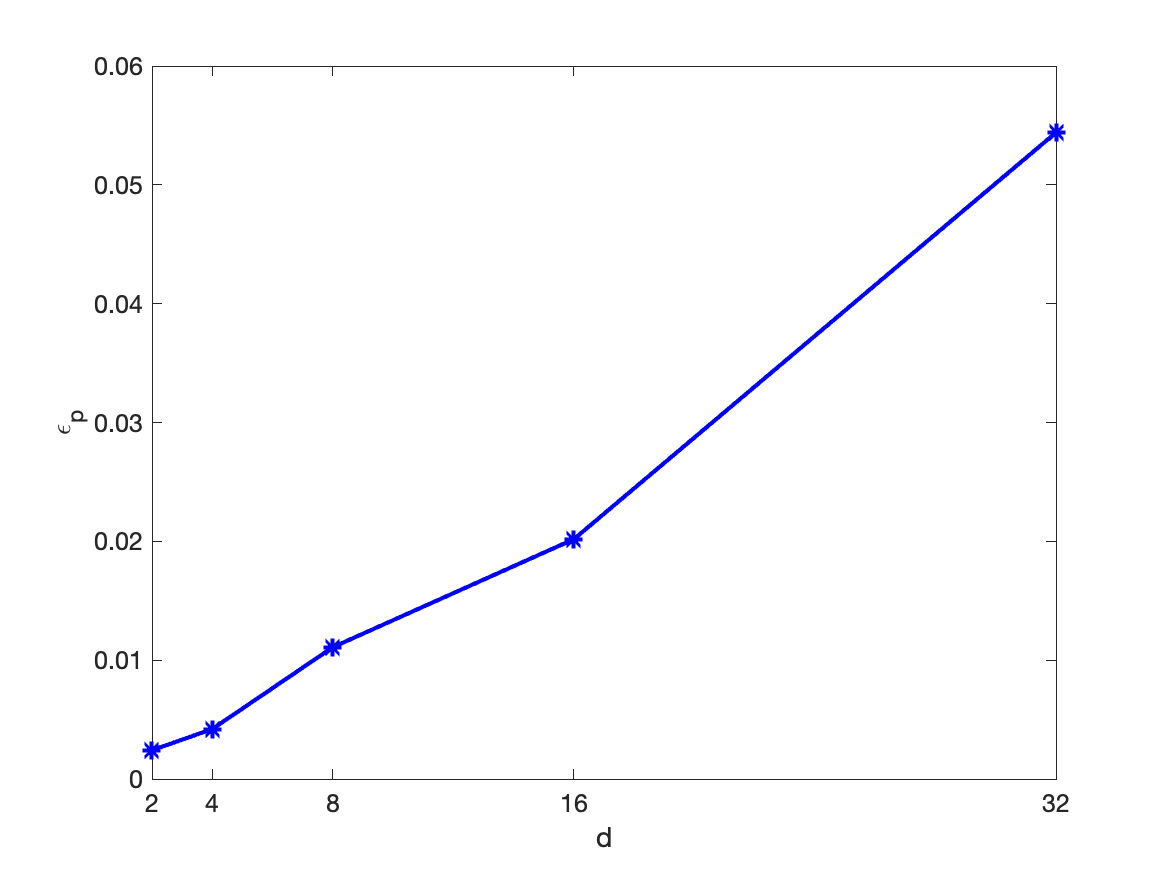}}
    \caption{Change of error with respect to $d$ in one-dimensional ferromagnetic Ising model with $\beta=0.6$.  }\label{fig:error_changed_ferro}
\end{figure}

\subsubsection{Antiferromagnetic Ising model}
Another important case is the antiferromagnetic example, where the ground-truth distribution has the same form
$$p^*(x_1,\cdots,x_d) \propto \exp(-\beta \sum_{i,j=1}^d J_{ij}x_ix_j)$$
but the interaction coefficients $J_{ij}$ are positive valued:
$$J_{ij}=
\begin{cases}
1/2, & |i-j|=1 \\
1/6, & |i-j|=2 \\
0, & \texttt{otherwise}
\end{cases}$$

Figure~\ref{fig:error_1d_antiferro} showcases the results obtained by hierarchical tensor network. Again, the error decreases in terms of $N$ at a rate of $1/\sqrt{N}$. However, when compared to the ferromagnetic Ising model, the error is larger. To gain a better understanding of this difference, we could visualize the empirical distribution. In Figure~\ref{fig:empirical_1d_antiferro}, it is evident from the antiferromagnetic type distribution that it is less concentrated and has more small local peaks compared to the ferromagnetic example at the same temperature, thus more samples and more detailed sketches are required to accurately recover the antiferromagnetic case.
Additionally,  the distribution of antiferromagnetic Ising model concentrates on two states, each of which have mixed $+1$ and $-1$ variables and are visualized in Figure~\ref{fig:empirical_1d_antiferro}. Furthermore, states around the two main states occur with non-negligible probability. This could explain why the error of antiferromagnetic Ising model would be larger than the one in ferromagnetic Ising model.

\begin{figure}[!htb]
    \centering  {\includegraphics[width=0.6\textwidth]{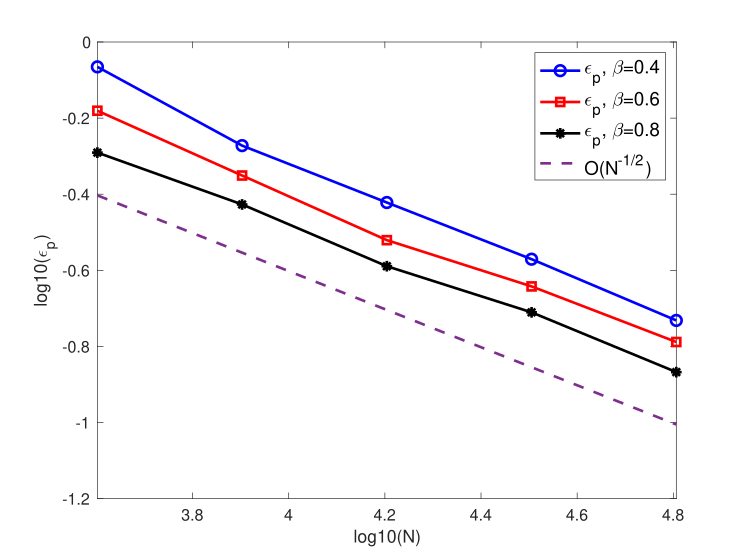}}
    \caption{Error for next neighbor antiferromagnetic Ising model with $d=16$. Blue, red, and black curves represent cases with $\beta=0.4$, $\beta=0.6$, and $\beta=0.8$, respectively. Dashed curve reflects a reference curve $O(N^{-1/2})$.}\label{fig:error_1d_antiferro}
\end{figure}

\begin{figure}[!htb]
    \centering
    \subfloat[Ferromagnetic Ising model.]{{\includegraphics[width=0.48\textwidth]{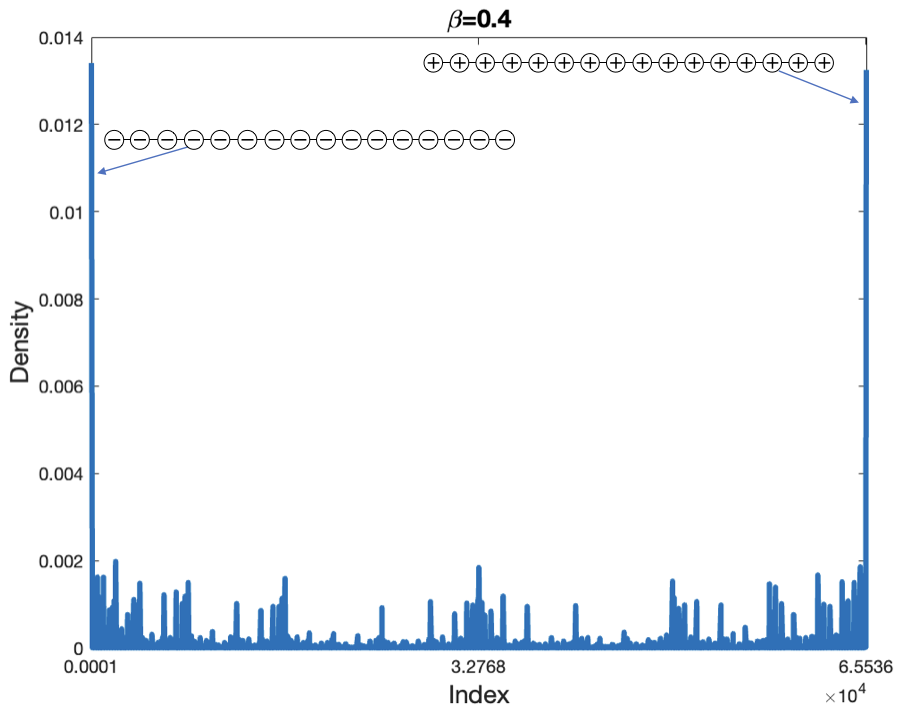}}}
    \quad 
    \subfloat[Antiferromagnetic Ising model.]
    {{\includegraphics[width=0.48\textwidth]{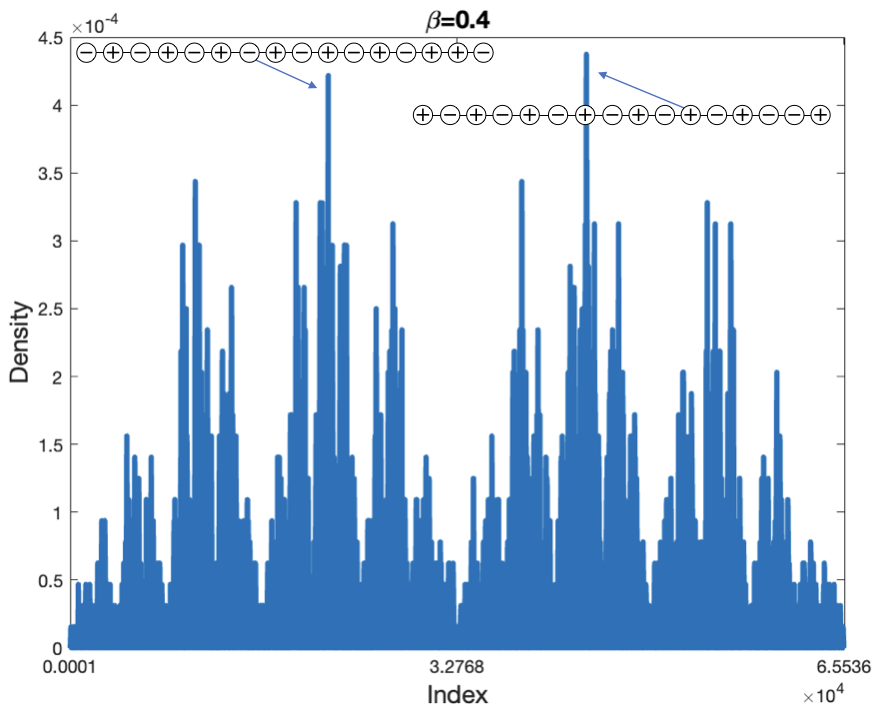}}}
    \caption{The empirical distribution with $N=64000$, $d=16$ and $\beta=0.4$ of ferromagnetic Ising model (Left) and antiferromagnetic Ising model (Right). Four representative states are included in the figure. "$+$" and "$-$" are corresponding to $x_i=+1$ and $x_i=-1$, respectively.   }\label{fig:empirical_1d_antiferro}   
\end{figure}

\subsection{Two-dimensional Ising model}
The next example is two-dimensional Ising model.  Here we only consider the interaction between nearest neighbor. The probability distribution is given by the following:
$$p^*(\{x_{i,j}\}_{i,j=1,\cdots,d})\propto \exp\left(-\beta 
\sum_{i=1}^{d}\sum_{j=1}^d(Jx_{i,j}x_{i+1,j}+Jx_{i,j}x_{i,j+1})\right)$$
where $\beta>0$. We consider periodic boundary condition in the problem, which means $x_{i,d+1}=x_{i,1}$ and $x_{d+1,j}=x_{1,j}$ for $i,j=1,\cdots,d$.  \par 

\subsubsection{Ferromagnetic Ising model}
For the ferromagnetic Ising model we take $J=-1$. The left figure in Figure~\ref{fig:2d_result_ferro} displays results for different numbers of samples at different temperatures, for a $4\times 4$ Ising model, while the right figure gives the results for an $8\times 8$ system. Similar to the one-dimensional Ising model, the error decreases in terms of $N$ at a rate of $1/\sqrt{N}$. For two-dimensional ferromagnetic Ising model, we visualize five typical samples in example $\beta=0.4$ of $4\times 4$ domain. In Figure~\ref{fig:visualize_2d_ferro}, we find that samples from the distribution would concentrate on two states: one state has all positive signs and the other has all negative signs.

\begin{figure}[!htb]
    \centering
    \subfloat[$4\times 4$ domain.]{{\includegraphics[width=0.48\textwidth]{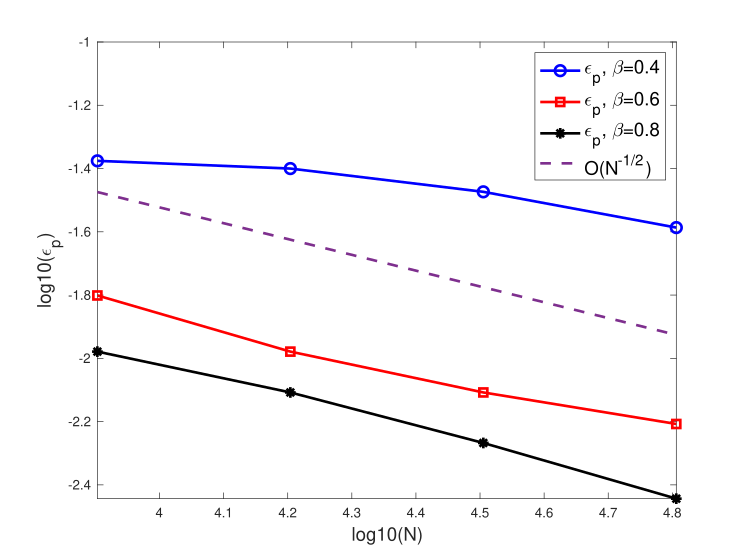}}}
    \quad 
    \subfloat[$8\times 8$ domain.]
    {{\includegraphics[width=0.48\textwidth]{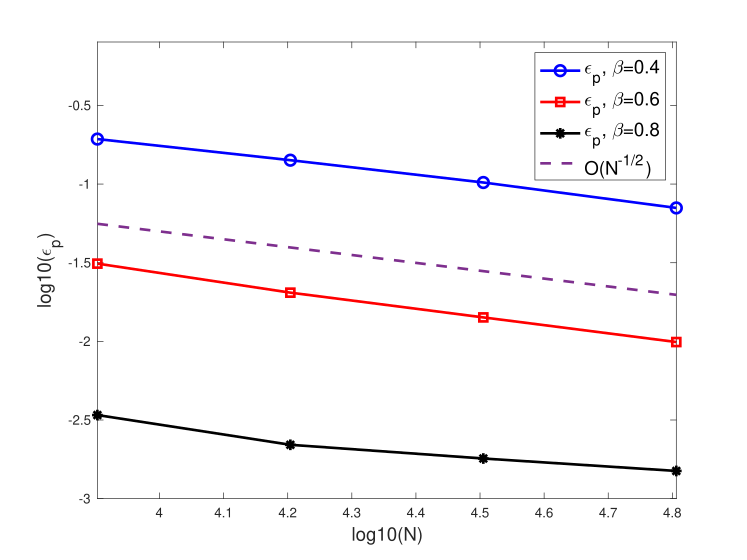}}}
    \caption{Error of nearest neighbor Ising model in several temperature scenarios. Blue, red, and black curves represent cases with $\beta=0.4$, $\beta=0.6$, and $\beta=0.8$, respectively. Dashed curve reflects a reference curve $O(N^{-1/2})$. Left: $4\times 4$ domain; Right: $8\times 8$ domain. }\label{fig:2d_result_ferro}   
\end{figure}

\begin{figure}[!htb]
    \centering
    {\includegraphics[width=1.0\textwidth]{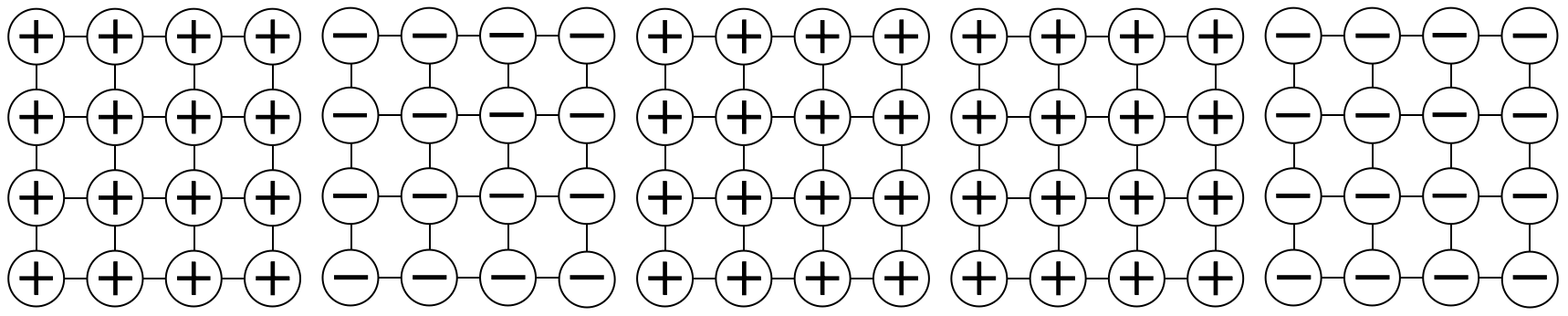}}
    \caption{Visualization of five samples in $4\times 4$ ferromagnetic Ising model with $\beta=0.4$. Here "$+$" represents $x_{i,j}=+1$ and "$-$" represents $x_{i,j}=-1$ for $i,j=1,\cdots,4$. }\label{fig:visualize_2d_ferro}
\end{figure}

\subsubsection{Antiferromagnetic Ising model}
In this section, we study the antiferromagnetic two-dimensional Ising model. The probability distribution can be expressed as:
$$p^*(\{x_{i,j}\}_{i,j=1,\cdots,d})\propto \exp\left(-\beta 
\sum_{i=1}^{d}\sum_{j=1}^d(Jx_{i,j}x_{i+1,j}+Jx_{i,j}x_{i,j+1})\right)$$
with $J=1$. In Figure~\ref{fig:visualize_2d_antiferro}, five typical samples from this type of distribution are demonstrated, which shows oscillatory sign pattern. This stands in contrast to the ferromagnetic Ising model. Figure~\ref{fig:2d_result_antiferro} presents the results obtained in both the $4\times 4$ and $8\times 8$ models. The algorithm exhibits decent accuracy with a sufficient number of samples. 

\begin{figure}[!htb]
    \centering
    {\includegraphics[width=1.0\textwidth]{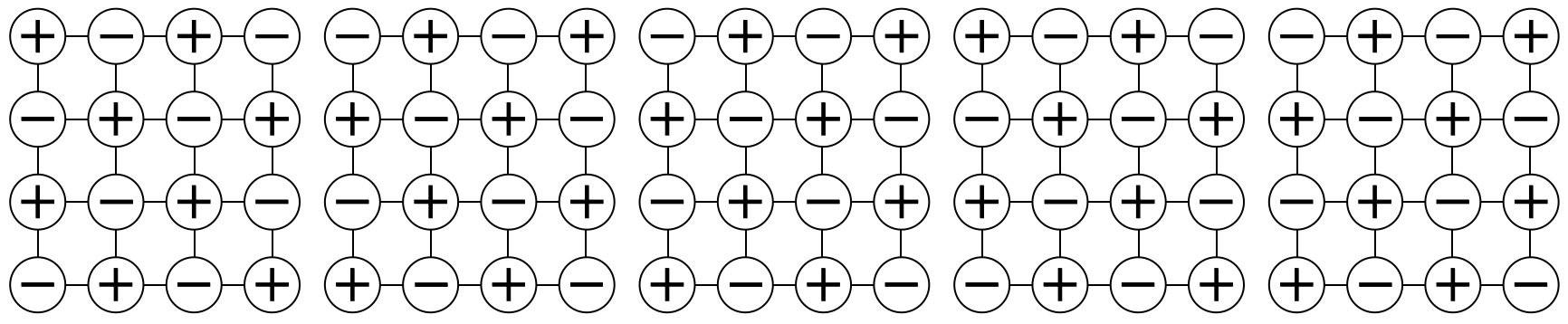}}
    \caption{Visualization of five samples in $4\times 4$ antiferromagnetic Ising model with $\beta=0.4$. Here "$+$" represents $x_{i,j}=+1$ and "$-$" represents $x_{i,j}=-1$ for $i,j=1,\cdots,4$. }\label{fig:visualize_2d_antiferro}
\end{figure}

\begin{figure}[!htb]
    \centering
    \subfloat[$4\times 4$ domain.]{{\includegraphics[width=0.48\textwidth]{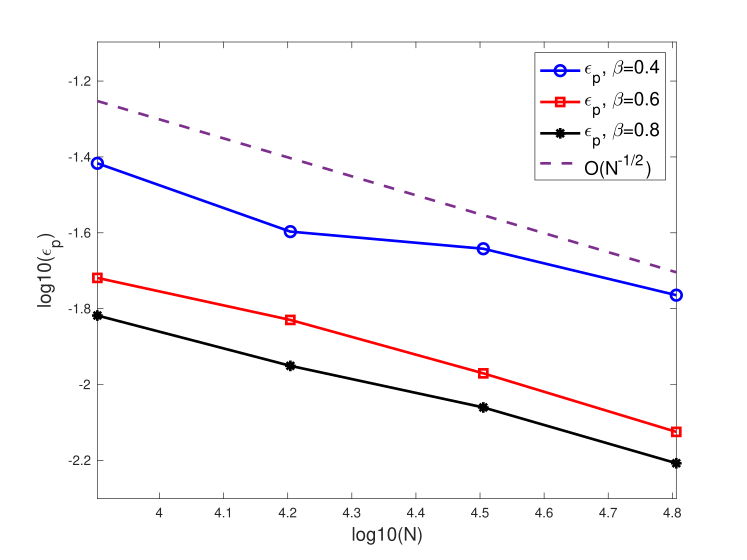}}}
    \quad 
    \subfloat[$8\times 8$ domain.]
    {{\includegraphics[width=0.48\textwidth]{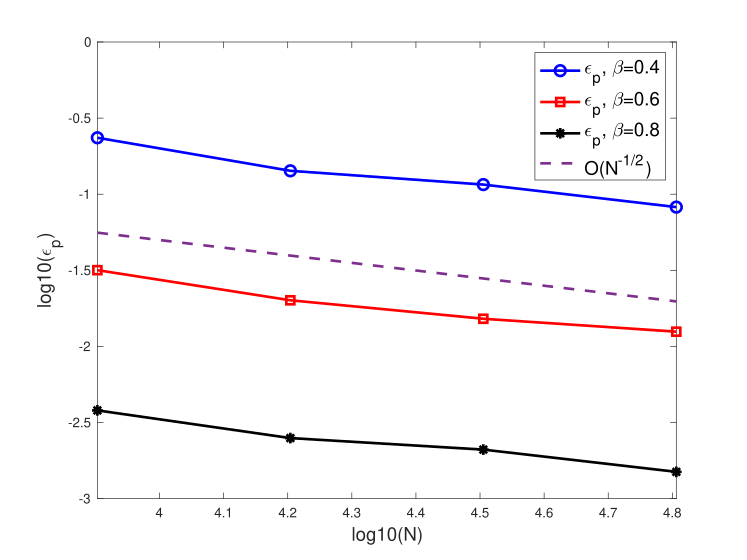}}}
    \caption{Error of nearest neighbor antiferromagnetic Ising model with several temperatures. Blue, red, and black curves represent cases with $\beta=0.4$, $\beta=0.6$, and $\beta=0.8$, respectively. Dashed curve reflects a reference curve $O(N^{-1/2})$. Left: $4\times 4$ domain; Right: $8\times 8$ domain. }\label{fig:2d_result_antiferro}   
\end{figure}

\subsubsection{Choice of rank}\label{sec: rank truncation}
The choice of rank in this algorithm is crucial as it directly affects the model's complexity, i.e., the number of parameters in the model. Like many machine learning methods, there is a trade-off when selecting the rank, the hyperparameters in this model: If the rank is too small, the model may not be able to fit the distribution well due to a lack of representation power. On the other hand, if the rank is too large, it may lead to overfitting  which results in a large estimation error.

To illustrate the impact of rank, we consider a $4\times 4$ ferromagnetic Ising model with $\beta=0.2$. Our first experiment is to study the effect of rank at the top level $r^{(0)}_1$ (rank of $G^{(0)}_1$) on the estimation error with a fixed number of samples $N$. We compare the  error  $\epsilon_p = \|\tilde p-p^*\|_F/\|p^*\|_F$ with the relative approximation error $\epsilon_{{\texttt{approx}}}=\|\ccA(p^*)-p^*\|_F/\|p^*\|_F$, i.e. the error when we replace $\hat p$ with the ground truth $p^*$ when applying $\mathcal{A}(\cdot)$. In Figure~\ref{fig:2d_d4_beta01_rank}, we show that the best rank that achieves the smallest error $\epsilon_p$ for each $N$, increases as $N$ increases. When $N$ is large, the error of $\ccA(\hat p)$ is converging to the approximation error $\epsilon_{{\texttt{approx}}}$. We can explain this trend as follows: when $N$ is small, by introducing a bias error via using a smaller rank, we can reduce the variance in estimation. In contrast, with large number of samples, we can afford to use a large rank to reduce the approximation error.
\begin{figure}[!htb]
    \centering   \includegraphics[width=0.6\textwidth]{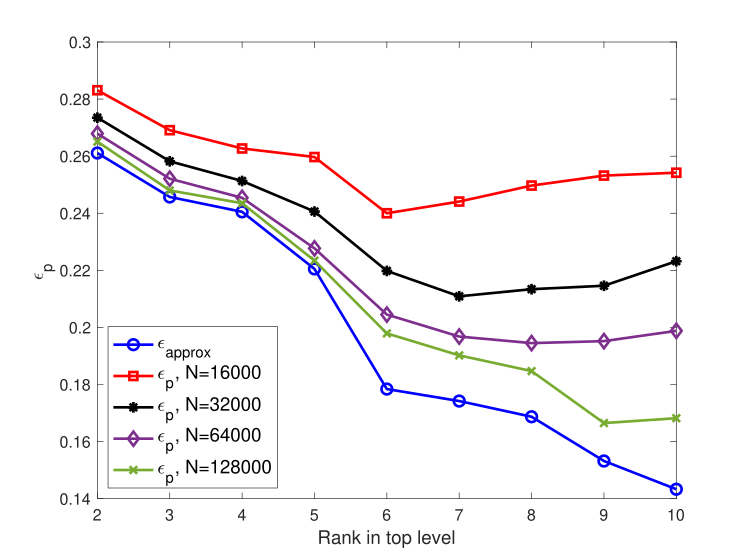} \quad
    \caption{The influence of truncated rank in top level on error in $4\times 4$ ferromagnetic Ising model with $\beta=0.2$. The approximation error $\epsilon_{{\texttt{approx}}}$ is shown in the line with circle marker; errors of empirical distribution with $N=16000$, $N=32000$, $N=64000$ and $N=128000$ are shown in line with square, star, diamond and cross marker, respectively. }\label{fig:2d_d4_beta01_rank}
\end{figure}

Our second experiment analyzes how the error changes with respect to $N$ for fixed ranks at the top level. Using the same example as the first test, we depict the change in the total error $\epsilon_p$ with respect to the number of samples for two choices of rank, $6$ and $10$, respectively, in Figure~\ref{fig:2d_d4_beta02_rank_N}. It is apparent that the relative approximation error is smaller with $\text{rank}=10$ than with $\text{rank}=6$. However, when we compute the estimation error ($\epsilon_p - \epsilon_{\texttt{approx}}$), the case of $\text{rank}=6$ shows faster convergence. We can explain this observation using the bias-variance trade-off. The bias error is reflected by $\epsilon_{{\texttt{approx}}}$, which depends solely on the model complexity and problem difficulty. In this figure, although $\text{rank}=10$ provides smaller bias error, the total error is larger than the $\text{rank}=6$ case when number of samples is small. This again shows that with limited number of samples, one might want to commit a bias error to reduce the variance in the estimator $\tilde{p} =\mathcal{A}(\hat{p})$.

\begin{figure}[!htb]
    \centering   \includegraphics[width=0.6\textwidth]{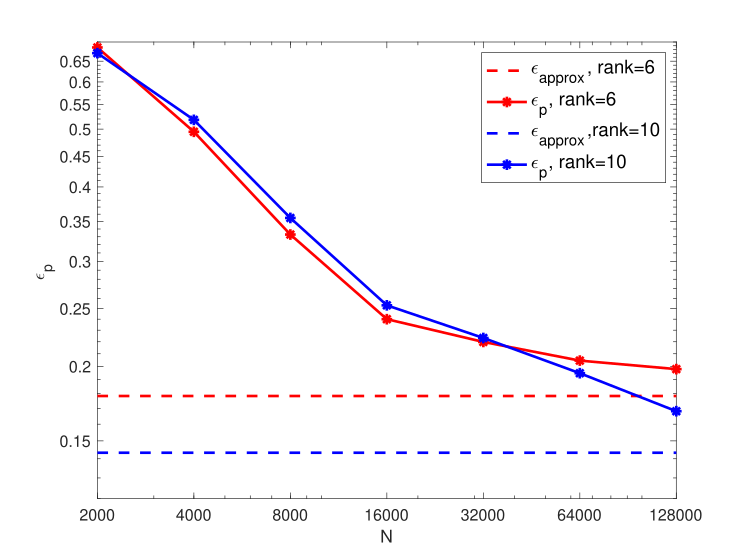} \quad
    \caption{Change of error with respect to number of samples for two choices of rank in top level in $4\times 4$ ferromagnetic Ising model with $\beta=0.2$. The benchmark line given by relative approximation error is shown as red and blue dashed line for rank $6$ and $10$, respectively. The convergence curves are shown as red and blue star lines for rank $6$ and $10$, respectively. }\label{fig:2d_d4_beta02_rank_N}
\end{figure}

% \newpage
\section{Conclusion}\label{sec:conclusion}

We present an optimization-free method for constructing a hierarchical tensor-network that approximates high-dimensional probability density using empirical distributions. Our approach involves sketching techniques from randomized linear algebra to form the tensor-network, with a computational complexity of $O(Nd\log d)$. We further demonstrate the success of the algorithm in several numerical examples concerning Ising-like models. 

\paragraph{Acknowledgment}
Y. Khoo acknowledges partial supports of NSF DMS-2339439, DOE DE-SC0022232, DARPA The Right
Space HR0011-25-9-0031, and a Sloan research fellowship.

% \section{Appendix}
\appendix
\section{Preliminaries of Theorem~\ref{thm:main}}\label{sec:Preliminaries}
We first provide a few useful lemmas and corollaries for conducting the analysis.
\begin{lemma}\label{lemma:wedin}(Theorem 4.1 in \cite{wedin1973perturbation})
    Suppose two matrices $A,B$ have the same rank $r(A)=r(B)$, then   
   $$ \|A^\dagger - B^\dagger\|_2\leq 3\|A^\dagger\|_2\|B^\dagger\|_2\|A-B\|_2$$
\end{lemma}

\begin{cor}(Corollary to Matrix Bernstein inequality, cf. Corollary 6.2.1 in \cite{tropp2015introduction})\label{cor:bernstein} Let $A^* \in \mathbb{R}^{m_1\times m_2}$ be a matrix, and let $\{A^{\{j\}}\in \mathbb{R}^{m_1\times m_2}\}_{j=1}^N$ be a sequence of i.i.d. matrices with $\mathbb{E}[A^{\{j\}}]=A^*$. Denote
$\hat{A}=\frac{1}{N}\sum_{j=1}^N A^{\{j\}}$. Suppose there exists a constant $M$ such that $\|A^{\{j\}}\|_2\leq M$ for every $j$, then the following probability inequality holds for any $t>0$:
$$\mathbb{P}[\|\hat{A}-A^*\|_2\geq t]\leq (m_1+m_2)\exp(\frac{-Nt^2/2}{M^2+2Mt/3})$$
\end{cor}

\begin{cor}\label{cor:tensor_norm}
Assume $D\in \bbR^{r_1\times r_2 \times r_3}$ is a three-dimensional tensor and $A\in \bbR^{r_1 \times n_1}, B\in \bbR^{r_2\times n_2}, C\in \bbR^{n_1\times n_2\times r_3}$ satisfies $D(:,:,\gamma)=AC(:,:,\gamma)B^T$ for $\gamma\in [r_3]$. Then the following upper bound holds
$$\|D\|_F \leq \|A\|_2\|B\|_2\|C\|_F$$
\end{cor}
\begin{proof}
For each $\gamma$, $\|D(:,:,\gamma)\|_F = \|AC(:,:,\gamma)B^T\|_F \leq \|A\|_2\|C(:,:,\gamma)B^T\|_F=\|A\|_2\|BC(:,:,\gamma)^T\|_F \leq \|A\|_2\|B\|_2\|C(:,:,\gamma)^T\|_F=\|A\|_2\|B\|_2\|C(:,:,\gamma)\|_F$. Then the following inequality holds,
$$\|D\|_F^2=\sum_{\gamma}\|D(:,:,\gamma)\|^2_F
\leq \sum_{\gamma}\|A\|^2_2\|B\|^2_2\|C(:,:,\gamma)\|^2_F=\|A\|^2_2\|B\|^2_2\|C\|^2_F.$$
Therefore, $\|D\|_F \leq \|A\|_2\|B\|_2\|C\|_F$.
\end{proof}

\section{Proof of Theorem~\ref{thm:main}}\label{sec:mainproof}
We first summarize the organization of this subsection. First, we  study the perturbation error of tensor core $\|\hat G^\bl_k - G^{\bl*}_k\|_F$ in each level. Then based on the hierarchical structure of the tensor-network, we could recursively bound the error $\|\tilde p^\blm_k -  p^{\blm*}_k\|_F$ based on lower level error $\|\tilde p^\bl_k -  p^{\bl*}_k\|_F$ in the hierarchical tensor-network structure, and thus bound the final error $\|\tilde{p}-p^*\|_F$ in the end. We split the whole proof into three lemmas. \par 

\begin{lemma}\label{lemma:errorG} 
$\|\hat{G}_k^{(l-1)} - G_k^{(l-1)*}\|_F \leq  \kappa_{G}\epsilon_A$ for $k\in[2^l],l\in[L]$ where $\kappa_{G}=c^2_{\hat{A}^\dagger}   + 6c_{A^\dagger}c^2_{\hat{A}^\dagger}c_p + 6c^2_{A^\dagger}c_{\hat{A}^\dagger}c_p.$
\end{lemma}
\begin{proof}
Based on \eqref{eq:solveG}, $G_k^{(l-1)*}$ for $l=1,\cdots,L$ takes the form 
\begin{equation}\label{eq:Gproof}
    G_k^{(l-1)*} = \left(\ccP_{r^{(l)}}(A^{(l)*}_{2k-1})\right)^\dagger B_k^{(l-1)*}\left(\ccP_{r^{(l)}}(A_{2k}^{(l)*})^T\right)^\dagger = \left(A^{(l)*}_{2k-1}\right)^\dagger B_k^{(l-1)*}\left(A_{2k}^{(l)*T}\right)^\dagger, k\in [2^{l-1}].
\end{equation}
The second equality is due to Proposition~\ref{prop:trim_rankr}. Similarly, we have a corresponding representation $\hat{G}_k^{(l-1)}$ for empirical distribution $\hat{p}$:
\begin{equation}\label{eq:tildeGproof}
    \hat{G}_k^{(l-1)} = \left(\ccP_{r^{(l)}}(\hat{A}^{(l)}_{2k-1})\right)^\dagger \hat{B}_k^{(l-1)}\left(\ccP_{r^{(l)}}(\hat{A}_{2k}^{(l)})^T\right)^\dagger, k\in [2^{l-1}].
\end{equation}
To analyze the error of $\hat{G}_k^{(l-1)}$, we need to analyze the error of $\left(\ccP_{r^{(l)}}(\hat A^{(l)}_{2k-1})\right)^\dagger,\left(\ccP_{r^{(l)}}(\hat A^{(l)}_{2k})\right)^\dagger$ and $\hat B_k^{(l-1)}$, respectively. Here we use Lemma \ref{lemma:wedin} for $\left(\ccP_{r^{(l)}}(\hat A^{(l)}_{i})\right)^\dagger$, ($i = 2k-1, 2k$) 
\begin{equation}\label{eq:Adagger bound}
\|\left(A^{(l)*}_i\right)^\dagger- \left(\ccP_{r^{(l)}}(\hat{A}^{(l)}_i)\right)^\dagger\|_2 \leq 
3\|\left(A^{(l)*}_i\right)^\dagger\|_2\|\left(\ccP_{r^{(l)}}(\hat{A}^{(l)}_i)\right)^\dagger\|_2 \|A^{(l)*}_i - \ccP_{r^{(l)}}(\hat{A}^{(l)}_i) \|_2,
\end{equation}
where 
\begin{align}
\|A^{(l)*}_i - \ccP_{r^{(l)}}(\hat{A}^{(l)}_i) \|_2 &=
\|A^{(l)*}_i - \hat{A}^{(l)}_i + \hat{A}^{(l)}_i - \ccP_{r^{(l)}}(\hat{A}^{(l)}_i) \|_2 \nonumber \\ 
&\leq \|A^{(l)*}_i - \hat{A}^{(l)}_i\|_2 + 
\|\hat{A}^{(l)}_i - \ccP_{r^{(l)}}(\hat{A}^{(l)}_i) \|_2 \leq 2\|A^{(l)*}_i - \hat{A}^{(l)}_i \|_2 \leq 2\epsilon_A.
\end{align}
The second inequality follows the fact that $\ccP_{r^{(l)}}(\hat{A}^{(l)}_i)$ is the best rank-$r^{(l)}$ approximation of $\hat{A}^{(l)}_i$ in terms of spectral norm and $\text{rank}(A^{(l)*}_i)=r^{(l)}$. Plugging this into \eqref{eq:Adagger bound}, we get 
\begin{equation}\label{eq:perturbA}
   \|\left(A^{(l)*}_i\right)^\dagger- \left(\ccP_{r^{(l)}}(\hat{A}^{(l)}_i)\right)^\dagger\|_2 \leq 6c_{A^\dagger}c_{\hat{A}^\dagger}\epsilon_A, 
\end{equation}
where we use the definition in \eqref{eq:constant def} for bounding $\|\left(A_i^{(l)*}\right)^\dagger\|_2$ and $\|\left(\ccP_{r^{(l)}}(\hat{A}^{(l)}_i)\right)^\dagger\|_2$.

Next, we upper bound the error of $\hat{B}_k^{(l-1)}$ (defined via \eqref{eq:formulaB}, as mentioned in the introduction of this section):
\begin{align}
    \|\hat{B}_k^{(l-1)}-B_k^{(l-1)*}\|_F  &= 
    \|S^{(l)T}_{2k-1}\hat{p}_k^{(l-1)}{S^{(l)}_{2k}} - S^{(l)T}_{2k-1}p_k^{(l-1)*}{S^{(l)}_{2k}}  \|_F  
    \leq c^2_S \|\hat{p}_k^{(l-1)}-p_k^{(l-1)*}\|_F 
    \nonumber \\
    &\leq c^2_S\|\left(S_k^{(l-1)T}\right)^\dagger\|_2 \|\hat{A}_k^{(l-1)}-A_k^{(l-1)*} \|_F
    \leq c^2_Sc_{S^\dagger}\epsilon_{A} = \epsilon_{A},  
\end{align}
where the first inequality is due to Corollary~\ref{cor:tensor_norm} and the second inequality is from \eqref{eq:formulaA}, the definition of $\hat{A}_k^{(l-1)}$ and $A_k^{(l-1)*}$. Here  $c_S=c_{S^\dagger}=1$ since $S_{2k-1}^{(l)},S_{2k}^{(l)},S_{k}^{(l-1)}$ are orthogonal matrices in the assumption of the theorem. 
Besides,
\begin{equation}\label{eq:boundB*}
    \|B_k^{(l-1)*}\|_F = \|S^{(l)T}_{2k-1}p_k^{(l-1)*}{S^{(l)}_{2k}}  \|_F \leq c^2_S \|p_k^{(l-1)*}\|_F \leq c_p,
\end{equation}
where the first inequality is from Corollary~\ref{cor:tensor_norm}. 
Therefore, we could have the following bound:
\begin{align} 
    \|\hat{G}_k^{(l-1)} &- G_k^{(l-1)*}\|_F   = 
    \|\left(\ccP_{r^{(l)}}(\hat{A}^{(l)}_{2k-1})\right)^{\dagger}\hat{B}_k^{(l-1)}\left(\ccP_{r^{(l)}}(\hat{A}^{(l)}_{2k})^T\right)^{\dagger} - \left(A_{2k-1}^{(l)*}\right)^{\dagger}B_k^{(l-1)*}\left({A_{2k}^{(l)*T}}\right)^{\dagger}\|_F \nonumber \\      
    & \leq  
    \|\left(\ccP_{r^{(l)}}(\hat{A}^{(l)}_{2k-1})\right)^{\dagger}\hat{B}_k^{(l-1)}\left(\ccP_{r^{(l)}}(\hat{A}^{(l)}_{2k})^T\right)^{\dagger}-\left(\ccP_{r^{(l)}}(\hat{A}^{(l)}_{2k-1})\right)^{\dagger}B_k^{(l-1)*}\left(\ccP_{r^{(l)}}(\hat{A}^{(l)}_{2k})^T\right)^{\dagger}\|_F \nonumber \\ 
    &+ 
    \|\left(\ccP_{r^{(l)}}(\hat{A}^{(l)}_{2k-1})\right)^{\dagger}B_k^{(l-1)*}\left(\ccP_{r^{(l)}}(\hat{A}^{(l)}_{2k})^T\right)^{\dagger}-\left(\ccP_{r^{(l)}}(\hat{A}^{(l)}_{2k-1})\right)^{\dagger}B_k^{(l-1)*}\left(A^{(l)*T}_{2k}\right)^{\dagger}\|_F \nonumber \\ 
    & + 
     \|\left(\ccP_{r^{(l)}}(\hat{A}^{(l)}_{2k-1})\right)^{\dagger}B_k^{(l-1)*}\left(A^{(l)*T}_{2k}\right)^{\dagger}-\left(A^{(l)*}_{2k-1}\right)^{\dagger}B_k^{(l-1)*}\left(A^{(l)*T}_{2k}\right)^{\dagger} \|_F \nonumber \\ 
    & \leq  c^2_{\hat{A}^\dagger}\|\hat{B}_k^{(l-1)} - B_k^{(l-1)*} \|_F   +  c_{\hat{A}^\dagger} c_p \|\left(\ccP_{r^{(l)}}(\hat{A}^{(l)}_{2k})^T\right)^{\dagger} - \left(A^{(l)*T}_{2k}\right)^{\dagger}  \|_2 \nonumber \\ 
    &+ c_{A^\dagger} c_p \|\left(\ccP_{r^{(l)}}(\hat{A}^{(l)}_{2k-1})\right)^{\dagger} -\left(A^{(l)*}_{2k-1}\right)^{\dagger}\|_2 \nonumber \\ 
    &   \leq  (c^2_{\hat{A}^\dagger}   + 6c_{A^\dagger}c^2_{\hat{A}^\dagger}c_p + 6c^2_{A^\dagger}c_{\hat{A}^\dagger}c_p)\epsilon_A,
\end{align}
where the first equality is from \eqref{eq:Gproof} and \eqref{eq:tildeGproof}, the second inequality is based on Corollary~\ref{cor:tensor_norm}, and the third inequality is from \eqref{eq:perturbA}. Besides,
\begin{equation}
    \|G_k^{(l-1)*}\|_F = \| (A_{2k-1}^{(l)*})^\dagger B_k^{(l-1)*}({A_{2k}^{(l)*T}})^\dagger \|_F \leq c^2_{A^\dagger}c_p,
\end{equation}
following Corollary~\ref{cor:tensor_norm} and \eqref{eq:boundB*}. 
\end{proof}

The next step is to analyze the error of the hierarchical tensor-network in estimating $p^*$.
\begin{lemma}\label{lemma:error_p}
If $c^2_{A^\dagger}c_pc_{\tilde{p}}+c^2_{A^\dagger}c^2_p>1$, then $\|\tilde{p}-p^*\|_F \leq 
(c^2_{A^\dagger}c_pc_{\tilde{p}}+c^2_{A^\dagger}c^2_p)^L\left(1+\frac{c^2_{\tilde{p}}\kappa_{G}}{c^2_{A^\dagger}c_pc_{\tilde{p}}+c^2_{A^\dagger}c^2_p-1} \right)\epsilon_A$ where $\kappa_G$ is defined in Lemma \ref{lemma:errorG}.  
\end{lemma}

\begin{proof}
We want to prove that error $\|\tilde{p}_k^{(l)}-p_k^{(l)*}\|_F$ has the following form: 
\begin{equation}\label{eq:induction}
\|\tilde{p}_k^{(l)}-p_k^{(l)*}\|_F \leq \kappa_{p^\bl}\epsilon_A,
\end{equation}
for $l=0,\cdots,L$ with some $\kappa_{p^\bl}$. Applying this to $\tilde{p}_1^{(0)}=\tilde{p},p_1^{(0)*}=p^*$ we can obtain the desired conclusion. We obtain the form of $\kappa_{p^\bl}$ in \eqref{eq:induction} by induction. \par 
At $L$-th level, we have $\tilde{p}_k^{(L)}(\ccI_k^{(L)},:)=\hat{p}(\ccI_k^{(L)},:)T_k^{(L)}$ for $k\in[2^L]$ from \eqref{eq:recursion}. Therefore,  
\begin{equation}
    \|\tilde{p}_k^{(L)}-p_k^{(L)*}\|_F
    =\|\hat{p}(\ccI_k^{(L)},:)T_k^{(L)}-p^*(\ccI_k^{(L)},:)T_k^{(L)}\|_F \leq \|\left(S_k^{(L)}\right)^\dagger\|_2 \|\hat{A}_k^{(L)} - A_k^{(L)*} \|_F 
    \leq c_{S^\dagger}\epsilon_A  =\epsilon_A,
\end{equation}
where the first inequality is due to \eqref{eq:formulaA}, and the definitions of $\hat A_k^{(l)}$ and  $A_k^{(l)*}$ are given in the introduction of this section. Thus, $\|\tilde{p}_k^{(L)}-p_k^{(L)*}\|_F \leq \kappa_{p^{(L)}}\epsilon_A$ holds for any $k$ where $\kappa_{p^{(L)}}=1$.  \par 

Next, we derive the error \eqref{eq:induction} for $(l-1)$-th level, in terms of the the error of $l$-th level. Assuming $\|\tilde{p}_k^{(l)}-p_k^{(l)*}\|_F \leq \kappa_{p^\bl}\epsilon_A$ holds for $k\in[2^l]$. Based on \eqref{eq:recursion}, we have
\begin{equation}
    \tilde{p}_k^{(l-1)} = \tilde{p}_{2k-1}^{(l)}\hat{G}_k^{(l-1)}\tilde{p}_{2k}^{(l)T}, 
    p_k^{(l-1)*} = p_{2k-1}^{(l)*}G_k^{(l-1)*} p_{2k}^{(l)*T},k\in [2^{l-1}],
\end{equation} 
and
\begin{align}
  \|\tilde{p}_k^{(l-1)} - p_k^{(l-1)*}\|_F &= \|\tilde{p}^{(l)}_{2k-1}\hat{G}_k^{(l-1)}\tilde{p}^{(l)T}_{2k} - p^{(l)*}_{2k-1}G_k^{(l-1)*}p^{(l)*T}_{2k} \|_F \nonumber \\
  &\leq \|\tilde{p}^{(l)}_{2k-1}\hat{G}_k^{(l-1)}\tilde{p}^{(l)T}_{2k} - \tilde{p}^{(l)}_{2k-1}G_k^{(l-1)*}\tilde{p}^{(l)T}_{2k}\|_F \nonumber \\
  &+ \|\tilde{p}^{(l)}_{2k-1}G_k^{(l-1)*}\tilde{p}^{(l)T}_{2k} - p^{(l)*}_{2k-1}G_k^{(l-1)*}\tilde{p}^{(l)T}_{2k}\|_F \nonumber \\ 
  &+ \|p^{(l)*}_{2k-1}G_k^{(l-1)*}\tilde{p}^{(l)T}_{2k} - p^{(l)*}_{2k-1}G_k^{(l-1)*}p^{(l)*T}_{2k} \|_F  \nonumber \\
  &\leq c^2_{\tilde{p}}\|\hat{G}_k^{(l-1)} - G_k^{(l-1)*} \|_F + c^2_{A^\dagger}c_pc_{\tilde{p}}\|\tilde{p}^{(l)}_{2k-1}- p^{(l)*}_{2k-1}\|_F + c^2_{A^\dagger}c^2_p\|\tilde{p}^{(l)}_{2k}-p^{(l)*}_{2k}\|_F  \nonumber\\
  &\leq 
  (c^2_{\tilde{p}}\kappa_{G} + c^2_{A^\dagger}c_pc_{\tilde{p}}\kappa_{p^{(l)}}+c^2_{A^\dagger}c^2_p\kappa_{p^{(l)}})\epsilon_A, k\in[2^{l-1}], 
\end{align}
where the second inequality is from Corollary \ref{cor:tensor_norm} and the third inequality is from Lemma \ref{lemma:errorG}.
Therefore, $\|\tilde{p}_k^{(l-1)}-p_k^{(l-1)*}\|_F \leq \kappa_{p^{(l-1)}}\epsilon_A$ holds for  $\kappa_{p^{(l-1)}}=c^2_{\tilde{p}}\kappa_{G} + (c^2_{A^\dagger}c_pc_{\tilde{p}}+c^2_{A^\dagger}c^2_p)\kappa_{p^{(l)}}$. By induction, this equality holds for every $l$ and we have 
\begin{equation}
    \kappa_{p^{(l-1)}}+\frac{c^2_{\tilde{p}}\kappa_{G}}{c^2_{A^\dagger}c_pc_{\tilde{p}}+c^2_{A^\dagger}c^2_p-1}= (c^2_{A^\dagger}c_pc_{\tilde{p}}+c^2_{A^\dagger}c^2_p)\left(\kappa_{p^{(l)}}+\frac{c^2_{\tilde{p}}\kappa_{G}}{c^2_{A^\dagger}c_pc_{\tilde{p}}+c^2_{A^\dagger}c^2_p-1} \right).
\end{equation}
Furthermore, 
\begin{equation}
    \kappa_{p^{(0)}}+\frac{c^2_{\tilde{p}}\kappa_{G}}{c^2_{A^\dagger}c_pc_{\tilde{p}}+c^2_{A^\dagger}c^2_p-1}= (c^2_{A^\dagger}c_pc_{\tilde{p}}+c^2_{A^\dagger}c^2_p)^L\left(\kappa_{p^{(L)}}+\frac{c^2_{\tilde{p}}\kappa_{G}}{c^2_{A^\dagger}c_pc_{\tilde{p}}+c^2_{A^\dagger}c^2_p-1} \right),
\end{equation}
with the assumption $c^2_{A^\dagger}c_pc_{\tilde{p}}+c^2_{A^\dagger}c^2_p>1$, we have
\begin{align}
    \kappa_{p^{(0)}}&\leq  (c^2_{A^\dagger}c_pc_{\tilde{p}}+c^2_{A^\dagger}c^2_p)^L\left(\kappa_{p^{(L)}}+\frac{c^2_{\tilde{p}}\kappa_{G}}{c^2_{A^\dagger}c_pc_{\tilde{p}}+c^2_{A^\dagger}c^2_p-1} \right) \nonumber \\
    &=(c^2_{A^\dagger}c_pc_{\tilde{p}}+c^2_{A^\dagger}c^2_p)^L\left(1+\frac{c^2_{\tilde{p}}\kappa_{G}}{c^2_{A^\dagger}c_pc_{\tilde{p}}+c^2_{A^\dagger}c^2_p-1} \right).    
\end{align}
In the end, we have 
\begin{equation}
\|\tilde{p}-p^*\|_F \leq \kappa_{p^{(0)}}\epsilon_A=
(c^2_{A^\dagger}c_pc_{\tilde{p}}+c^2_{A^\dagger}c^2_p)^{\log_2{d}}\left(1+\frac{c^2_{\tilde{p}}\kappa_{G}}{c^2_{A^\dagger}c_pc_{\tilde{p}}+c^2_{A^\dagger}c^2_p-1} \right)\epsilon_A
\end{equation}
with $L=\log_2{d}$ mentioned in the introduction of this section. 
\end{proof}
In the following lemma, we state what the upper-bound $\epsilon_A$ is, which completes the proof for Theorem~\ref{thm:main}.
\begin{lemma}\label{lemma:epsilon_A}
For $0<\delta<1$, the following inequality for $\epsilon_A$ holds with probability at least $1-\delta$:
\begin{equation}
\epsilon_A \leq  \frac{3\sqrt{\tilde{r}_{\texttt{max}}}\log(\frac{2\tilde{r}_{\texttt{max}}d\log_2{d}}{\delta})}{\sqrt{N}}
\end{equation}
\end{lemma}
\begin{proof}
We first consider $\|\hat{A}_k^\bl - A_k^{(l)*}\|_F$ for each $k,l$. This can be bounded via spectral norm:
$\|\hat{A}_k^\bl - A_k^{(l)*}\|_F\leq \sqrt{\tilde{r}^\bl}\|\hat{A}_k^\bl - A_k^{(l)*}\|_2$. The error in spectral norm can be further bounded via matrix Bernstein inequality in Corollary \ref{cor:bernstein}. \par 

More specifically, by definition $\hat{A}_k^\bl$ takes the form
\begin{equation}
    \hat{A}_k^\bl =S_k^{(l)T}(\ccI_k^\bl,:)\hat{p}(\ccI_k^\bl;\ccJ_k^\bl)T_k^{(l)}(\ccJ_k^\bl,:) = \frac{1}{N}\sum_{j=1}^N S_k^{(l)}(\vy_{C_k^\bl}^{j},:)^T T_k^\bl(\vy_{C_k^\bl}^{j},:),
\end{equation}
where $\vy^j$'s are i.i.d. samples from $p^*$. The spectral norm of each term in the summation $\|S_k^{(l)}(\vy_{C_k^\bl}^{j},:)^T T_k^\bl(\vy_{C_k^\bl}^{j},:)\|_2$ is bounded by $1$ since $S_k^\bl$ and $T_k^\bl$ are orthogonal matrices. Moreover, $\mathbb{E}[S_k^{(l)}(\vy_{C_k^\bl}^{j},:)^T T_k^\bl(\vy_{C_k^\bl}^{j},:)]=A_k^{(l)*}$ since $\mathbb{E}[\hat{p}]=p^*$. Therefore, we could use matrix Bernstein inequality in Corollary \ref{cor:bernstein} and obtain
\begin{equation}
\mathbb{P}[\|\hat{A}_k^\bl-A_k^{(l)*}\|_2\geq t]\leq 2\tilde{r}^\bl \exp(\frac{-Nt^2/2}{1+2t/3})
\end{equation} 
for any $t>0$. We set $\tilde{\delta}=2\tilde{r}^\bl \exp(\frac{-Nt^2/2}{1+2t/3})$ and suppose $N$ is sufficient large such that $0<\tilde{\delta}<1$. Then $\|\hat{A}_k^\bl-A_k^{(l)*}\|_2< t$ 
holds with probability at least $1-\tilde{\delta}$. We now get $t$ in terms of $\tilde{\delta}$:
\begin{align}\label{eq:def t}
    t &= \frac{-\frac{2}{3}\log(\frac{\tilde{\delta}}{2\tilde{r}^\bl})+ \sqrt{\frac{4}{9}\log^2(\frac{\tilde{\delta}}{2\tilde{r}^\bl}) - 2N\log(\frac{\tilde{\delta}}{2\tilde{r}^\bl})} }{N}= \frac{\frac{2}{3}\log(\frac{2\tilde{r}^\bl}{\tilde{\delta}})+ \sqrt{\frac{4}{9}\log^2(\frac{2\tilde{r}^\bl}{\tilde{\delta}}) + 2N\log(\frac{2\tilde{r}^\bl}{\tilde{\delta}}) }}{N}\nonumber\\ 
    &\leq \frac{\frac{2}{3}\log(\frac{2\tilde{r}^\bl}{\tilde{\delta}})+ \frac{2}{3}\log(\frac{2\tilde{r}^\bl}{\tilde{\delta}}) + \sqrt{2N\log(\frac{2\tilde{r}^\bl}{\tilde{\delta}}) }}{N}  \leq \frac{(\frac{4}{3}+\sqrt{2})\log(\frac{2\tilde{r}^\bl}{\tilde{\delta}})}{\sqrt{N}}\leq \frac{3\log(\frac{2\tilde{r}^\bl}{\tilde{\delta}})}{\sqrt{N}}, \nonumber\\
\end{align}
where the first inequality follows $\sqrt{a+b}\leq \sqrt{a}+\sqrt{b}$ for $a,b>0$ and the second inequality follows the fact that $N>1$ and $\log(\frac{2\tilde{r}^\bl}{\tilde{\delta}})>1$ from $0<\tilde{\delta}<1$. 

Therefore, using such a $t$ in \eqref{eq:def t}, we can have the following bound: 
\begin{equation}
    \|\hat{A}_k^\bl - A_k^{(l)*}\|_F\leq \sqrt{\tilde{r}^\bl}\|\hat{A}_k^\bl - A_k^{(l)*}\|_2 \leq \frac{3\sqrt{\tilde{r}^\bl}\log(\frac{2\tilde{r}^\bl}{\tilde{\delta}})}{\sqrt{N}} \leq \frac{3\sqrt{\tilde{r}_{\texttt{max}}}\log(\frac{2\tilde{r}_{\texttt{max}}}{\tilde{\delta}})}{\sqrt{N}},
\end{equation}
which holds with probability at least $1-\tilde{\delta}$ (here we recall notation $\tilde{r}_{\texttt{max}}=\max_{l}\tilde{r}^{(l)}$). 

 At this point we have a bound for $\|\hat{A}_k^\bl - A_k^{(l)*}\|_F$ for a specific $k,l$. However, we need a uniform bound over all possible choice of $k,l$ where $k\in[2^l],l\in[L]$. There are at most $d\log_2{d}$ pairs of $k,l$, then 
\begin{align}
&\mathbb{P}\left[\max_{k,l}\|\hat{A}_k^\bl - A_k^{(l)*}\|_F \leq \frac{3\sqrt{\tilde{r}_{\texttt{max}}}\log(\frac{2\tilde{r}_{\texttt{max}}}{\tilde{\delta}})}{\sqrt{N}}\right] = 1- \mathbb{P}\left[\max_{k,l}\|\hat{A}_k^\bl - A_k^{(l)*}\|_F> \frac{3\sqrt{\tilde{r}_{\texttt{max}}}\log(\frac{2\tilde{r}_{\texttt{max}}}{\tilde{\delta}})}{\sqrt{N}}\right] \nonumber\\ 
&\geq 1 - \mathbb{P}\left[\bigcup_{k,l}\left\{\|\hat{A}_k^\bl - A_k^{(l)*}\|_F>\frac{3\sqrt{\tilde{r}_{\texttt{max}}}\log(\frac{2\tilde{r}_{\texttt{max}}}{\tilde{\delta}})}{\sqrt{N}}\right\}\right] \nonumber\\ 
&\geq 1 - \sum_{k,l}\mathbb{P}\left[\|\hat{A}_k^\bl - A_k^{(l)*}\|_F>\frac{3\sqrt{\tilde{r}_{\texttt{max}}}\log(\frac{2\tilde{r}_{\texttt{max}}}{\tilde{\delta}})}{\sqrt{N}}\right] \geq 1-\tilde{\delta}d\log_2{d}. 
\end{align}
Let $\delta=\tilde{\delta}d\log_2{d}$ and suppose $N$ is sufficient large such that $0<\delta<1$. Thus, the following inequality for $\epsilon_A=\max_{k,l}\|\hat{A}_k^\bl - A_k^{(l)*}\|_F$ holds with probability at least $1-\delta$:
\begin{equation}
    \epsilon_A \leq  \frac{3\sqrt{\tilde{r}_{\texttt{max}}}\log(\frac{2\tilde{r}_{\texttt{max}}d\log_2{d}}{\delta})}{\sqrt{N}}
\end{equation}
\end{proof}

\bibliographystyle{elsarticle-num}
\bibliography{ref}

@article{hist1,
  title={On optimal and data-based histograms},
  author={Scott, David W},
  journal={Biometrika},
  volume={66},
  number={3},
  pages={605--610},
  year={1979},
  publisher={Oxford University Press}
}

@article{hist2,
  title={Consistency of data-driven histogram methods for density estimation and classification},
  author={Lugosi, G{\'a}bor and Nobel, Andrew},
  journal={The Annals of Statistics},
  volume={24},
  number={2},
  pages={687--706},
  year={1996},
  publisher={Institute of Mathematical Statistics}
}

@article{kde1,
  title={Remarks on some nonparametric estimates of a density function},
  author={Rosenblatt, Murray},
  journal={The annals of mathematical statistics},
  pages={832--837},
  year={1956},
  publisher={JSTOR}
}

@article{kde2,
  title={On estimation of a probability density function and mode},
  author={Parzen, Emanuel},
  journal={The annals of mathematical statistics},
  volume={33},
  number={3},
  pages={1065--1076},
  year={1962},
  publisher={JSTOR}
}

@article{autoreg1,
  title={Neural autoregressive distribution estimation},
  author={Uria, Benigno and C{\^o}t{\'e}, Marc-Alexandre and Gregor, Karol and Murray, Iain and Larochelle, Hugo},
  journal={The Journal of Machine Learning Research},
  volume={17},
  number={1},
  pages={7184--7220},
  year={2016},
  publisher={JMLR. org}
}

@inproceedings{autoreg2,
  title={Made: Masked autoencoder for distribution estimation},
  author={Germain, Mathieu and Gregor, Karol and Murray, Iain and Larochelle, Hugo},
  booktitle={International conference on machine learning},
  pages={881--889},
  year={2015},
  organization={PMLR}
}

@article{autoreg3,
  title={Masked autoregressive flow for density estimation},
  author={Papamakarios, George and Pavlakou, Theo and Murray, Iain},
  journal={Advances in neural information processing systems},
  volume={30},
  year={2017}
}

@inproceedings{nf1,
  title={Variational inference with normalizing flows},
  author={Rezende, Danilo and Mohamed, Shakir},
  booktitle={International conference on machine learning},
  pages={1530--1538},
  year={2015},
  organization={PMLR}
}

@article{nf2,
  title={Density modeling of images using a generalized normalization transformation},
  author={Ball{\'e}, Johannes and Laparra, Valero and Simoncelli, Eero P},
  journal={arXiv preprint arXiv:1511.06281},
  year={2015}
}

@article{nf3,
  title={Density estimation using real nvp},
  author={Dinh, Laurent and Sohl-Dickstein, Jascha and Bengio, Samy},
  journal={arXiv preprint arXiv:1605.08803},
  year={2016}
}

@inproceedings{nf4,
  title={Flow-gan: Combining maximum likelihood and adversarial learning in generative models},
  author={Grover, Aditya and Dhar, Manik and Ermon, Stefano},
  booktitle={Proceedings of the AAAI conference on artificial intelligence},
  volume={32},
  number={1},
  year={2018}
}

@article{vae,
  title={Auto-encoding variational bayes},
  author={Kingma, Diederik P and Welling, Max},
  journal={arXiv preprint arXiv:1312.6114},
  year={2013}
}

@article{gan,
  title={Generative adversarial networks},
  author={Goodfellow, Ian and Pouget-Abadie, Jean and Mirza, Mehdi and Xu, Bing and Warde-Farley, David and Ozair, Sherjil and Courville, Aaron and Bengio, Yoshua},
  journal={Communications of the ACM},
  volume={63},
  number={11},
  pages={139--144},
  year={2020},
  publisher={ACM New York, NY, USA}
}

@article{tt,
  title={Tensor-train decomposition},
  author={Oseledets, Ivan V},
  journal={SIAM Journal on Scientific Computing},
  volume={33},
  number={5},
  pages={2295--2317},
  year={2011},
  publisher={SIAM}
}

@article{mps1,
  title={Matrix product state representations},
  author={Perez-Garcia, David and Verstraete, Frank and Wolf, Michael M and Cirac, J Ignacio},
  journal={arXiv preprint quant-ph/0608197},
  year={2006}
}

@article{mps2,
  title={Density matrix formulation for quantum renormalization groups},
  author={White, Steven R},
  journal={Physical review letters},
  volume={69},
  number={19},
  pages={2863},
  year={1992},
  publisher={APS}
}

@article{ttde1,
  title={Modeling sequences with quantum states: a look under the hood},
  author={Bradley, Tai-Danae and Stoudenmire, E Miles and Terilla, John},
  journal={Machine Learning: Science and Technology},
  volume={1},
  number={3},
  pages={035008},
  year={2020},
  publisher={IOP Publishing}
}

@article{ttde2,
  title={Unsupervised generative modeling using matrix product states},
  author={Han, Zhao-Yu and Wang, Jun and Fan, Heng and Wang, Lei and Zhang, Pan},
  journal={Physical Review X},
  volume={8},
  number={3},
  pages={031012},
  year={2018},
  publisher={APS}
}

@inproceedings{ttde3,
  title={Tensor-train density estimation},
  author={Novikov, Georgii S and Panov, Maxim E and Oseledets, Ivan V},
  booktitle={Uncertainty in Artificial Intelligence},
  pages={1321--1331},
  year={2021},
  organization={PMLR}
}

@article{tt-sketch,
  title={Generative modeling via tensor train sketching},
  author={Hur, Yoonhaeng and Hoskins, Jeremy G and Lindsey, Michael and Stoudenmire, E Miles and Khoo, Yuehaw},
  journal={arXiv preprint arXiv:2202.11788},
  year={2022}
}

@article{tropp2015introduction,
  title={An introduction to matrix concentration inequalities},
  author={Tropp, Joel A and others},
  journal={Foundations and Trends{\textregistered} in Machine Learning},
  volume={8},
  number={1-2},
  pages={1--230},
  year={2015},
  publisher={Now Publishers, Inc.}
}

@article{generalized_nystorm,
  title={Fast and stable randomized low-rank matrix approximation},
  author={Nakatsukasa, Yuji},
  journal={arXiv preprint arXiv:2009.11392},
  year={2020}
}

@article{wedin1973perturbation,
  title={Perturbation theory for pseudo-inverses},
  author={Wedin, Per-{\AA}ke},
  journal={BIT Numerical Mathematics},
  volume={13},
  pages={217--232},
  year={1973},
  publisher={Springer}
}

@article{halko2011finding,
  title={Finding structure with randomness: Probabilistic algorithms for constructing approximate matrix decompositions},
  author={Halko, Nathan and Martinsson, Per-Gunnar and Tropp, Joel A},
  journal={SIAM review},
  volume={53},
  number={2},
  pages={217--288},
  year={2011},
  publisher={SIAM}
}

@article{drautz2019atomic,
  title={Atomic cluster expansion for accurate and transferable interatomic potentials},
  author={Drautz, Ralf},
  journal={Physical Review B},
  volume={99},
  number={1},
  pages={014104},
  year={2019},
  publisher={APS}
}

@article{dusson2022atomic,
  title={Atomic cluster expansion: Completeness, efficiency and stability},
  author={Dusson, Genevieve and Bachmayr, Markus and Cs{\'a}nyi, G{\'a}bor and Drautz, Ralf and Etter, Simon and van der Oord, Cas and Ortner, Christoph},
  journal={Journal of Computational Physics},
  volume={454},
  pages={110946},
  year={2022},
  publisher={Elsevier}
}

@article{dolgov2020approximation,
  title={Approximation and sampling of multivariate probability distributions in the tensor train decomposition},
  author={Dolgov, Sergey and Anaya-Izquierdo, Karim and Fox, Colin and Scheichl, Robert},
  journal={Statistics and Computing},
  volume={30},
  pages={603--625},
  year={2020},
  publisher={Springer}
}

@article{oseledets2010tt,
  title={TT-cross approximation for multidimensional arrays},
  author={Oseledets, Ivan and Tyrtyshnikov, Eugene},
  journal={Linear Algebra and its Applications},
  volume={432},
  number={1},
  pages={70--88},
  year={2010},
  publisher={Elsevier}
}

@inproceedings{savostyanov2011fast,
  title={Fast adaptive interpolation of multi-dimensional arrays in tensor train format},
  author={Savostyanov, Dmitry and Oseledets, Ivan},
  booktitle={The 2011 International Workshop on Multidimensional (nD) Systems},
  pages={1--8},
  year={2011},
  organization={IEEE}
}

@article{steinlechner2016riemannian,
  title={Riemannian optimization for high-dimensional tensor completion},
  author={Steinlechner, Michael},
  journal={SIAM Journal on Scientific Computing},
  volume={38},
  number={5},
  pages={S461--S484},
  year={2016},
  publisher={SIAM}
}

@article{khoo2017efficient,
  title={Efficient construction of tensor ring representations from sampling},
  author={Khoo, Yuehaw and Lu, Jianfeng and Ying, Lexing},
  journal={arXiv preprint arXiv:1711.00954},
  year={2017}
}

@inproceedings{wang2017efficient,
  title={Efficient low rank tensor ring completion},
  author={Wang, Wenqi and Aggarwal, Vaneet and Aeron, Shuchin},
  booktitle={Proceedings of the IEEE International Conference on Computer Vision},
  pages={5697--5705},
  year={2017}
}

@article{geoga2020scalable,
  title={Scalable Gaussian process computations using hierarchical matrices},
  author={Geoga, Christopher J and Anitescu, Mihai and Stein, Michael L},
  journal={Journal of Computational and Graphical Statistics},
  volume={29},
  number={2},
  pages={227--237},
  year={2020},
  publisher={Taylor \& Francis}
}

@article{chen2023scalable,
  title={Scalable Physics-based Maximum Likelihood Estimation using Hierarchical Matrices},
  author={Chen, Yian and Anitescu, Mihai},
  journal={arXiv preprint arXiv:2303.10102},
  year={2023}
}

@article{shin2019hierarchical,
  title={A hierarchical optimization architecture for large-scale power networks},
  author={Shin, Sungho and Hart, Philip and Jahns, Thomas and Zavala, Victor M},
  journal={IEEE Transactions on Control of Network Systems},
  volume={6},
  number={3},
  pages={1004--1014},
  year={2019},
  publisher={IEEE}
}

@article{karsanina2018hierarchical,
  title={Hierarchical optimization: Fast and robust multiscale stochastic reconstructions with rescaled correlation functions},
  author={Karsanina, Marina V and Gerke, Kirill M},
  journal={Physical review letters},
  volume={121},
  number={26},
  pages={265501},
  year={2018},
  publisher={APS}
}

@article{chen2020multiscale,
  title={Multiscale semidefinite programming approach to positioning problems with pairwise structure},
  author={Chen, Yian and Khoo, Yuehaw and Lindsey, Michael},
  journal={arXiv preprint arXiv:2012.10046},
  year={2020}
}

@article{cao2019hierarchical,
  title={Hierarchical-block conditioning approximations for high-dimensional multivariate normal probabilities},
  author={Cao, Jian and Genton, Marc G and Keyes, David E and Turkiyyah, George M},
  journal={Statistics and Computing},
  volume={29},
  number={3},
  pages={585--598},
  year={2019},
  publisher={Springer}
}

@article{genton2018hierarchical,
  title={Hierarchical decompositions for the computation of high-dimensional multivariate normal probabilities},
  author={Genton, Marc G and Keyes, David E and Turkiyyah, George},
  journal={Journal of Computational and Graphical Statistics},
  volume={27},
  number={2},
  pages={268--277},
  year={2018},
  publisher={Taylor \& Francis}
}

@article{ambikasaran2015fast,
  title={Fast direct methods for Gaussian processes},
  author={Ambikasaran, Sivaram and Foreman-Mackey, Daniel and Greengard, Leslie and Hogg, David W and O’Neil, Michael},
  journal={IEEE transactions on pattern analysis and machine intelligence},
  volume={38},
  number={2},
  pages={252--265},
  year={2015},
  publisher={IEEE}
}

@article{chen2021scalable,
  title={Scalable Gaussian Process Analysis for Implicit Physics-Based Covariance Models},
  author={Chen, Yian and Anitescu, Mihai},
  journal={International Journal for Uncertainty Quantification},
  volume={11},
  number={6},
  year={2021},
  publisher={Begel House Inc.}
}

@article{saibaba2012efficient,
  title={Efficient methods for large-scale linear inversion using a geostatistical approach},
  author={Saibaba, Arvind K and Kitanidis, Peter K},
  journal={Water Resources Research},
  volume={48},
  number={5},
  year={2012},
  publisher={Wiley Online Library}
}

@article{lee2011cognitive,
  title={How cognitive modeling can benefit from hierarchical Bayesian models},
  author={Lee, Michael D},
  journal={Journal of Mathematical Psychology},
  volume={55},
  number={1},
  pages={1--7},
  year={2011},
  publisher={Elsevier}
}

@article{lawson2012bayesian,
  title={Bayesian point event modeling in spatial and environmental epidemiology},
  author={Lawson, Andrew B},
  journal={Statistical Methods in Medical Research},
  volume={21},
  number={5},
  pages={509--529},
  year={2012},
  publisher={Sage Publications Sage UK: London, England}
}

@article{shi2021parallel,
  title={Parallel algorithms for computing the tensor-train decomposition},
  author={Shi, Tianyi and Ruth, Maximilian and Townsend, Alex},
  journal={arXiv preprint arXiv:2111.10448},
  year={2021}
}

@article{tang2022generative,
  title={Generative modeling via tree tensor network states},
  author={Tang, Xun and Hur, Yoonhaeng and Khoo, Yuehaw and Ying, Lexing},
  journal={arXiv preprint arXiv:2209.01341},
  year={2022}
}

@article{orus2014advances,
  title={Advances on tensor network theory: symmetries, fermions, entanglement, and holography},
  author={Or{\'u}s, Rom{\'a}n},
  journal={The European Physical Journal B},
  volume={87},
  pages={1--18},
  year={2014},
  publisher={Springer}
}

@article{chen2023committor,
  title={Committor functions via tensor networks},
  author={Chen, Yian and Hoskins, Jeremy and Khoo, Yuehaw and Lindsey, Michael},
  journal={Journal of Computational Physics},
  volume={472},
  pages={111646},
  year={2023},
  publisher={Elsevier}
}

@article{qin2022error,
  title={Error Analysis of Tensor-Train Cross Approximation},
  author={Qin, Zhen and Lidiak, Alexander and Gong, Zhexuan and Tang, Gongguo and Wakin, Michael B and Zhu, Zhihui},
  journal={arXiv preprint arXiv:2207.04327},
  year={2022}
}
\end{document}